\documentclass[reqno,10pt]{amsart}
\usepackage{amsmath,amsfonts,amsthm,amssymb,mathtools,enumerate}
\usepackage[margin=1in]{geometry}
\usepackage[english]{babel}
\usepackage[hidelinks]{hyperref}
\hypersetup{
  pdftitle={Measure-valued free-energy minimizers for trapped bosons with repulsive Coulomb interaction},
  pdfauthor={Gi-Chan Bae and Jinmyoung Seok},
  pdfsubject={Variational analysis of a semiclassical trapped Bose gas},
  pdfkeywords={Bose--Einstein condensation, Coulomb energy, measure-valued minimizers, obstacle problem}
}
\usepackage{xcolor}
\usepackage{todonotes}

\usepackage{graphicx}
\usepackage{subcaption}

\numberwithin{equation}{section}

\newtheorem{theorem}{Theorem}[section]
\newtheorem{lemma}[theorem]{Lemma}
\newtheorem{prop}[theorem]{Proposition}
\newtheorem{corollary}[theorem]{Corollary}
\newtheorem{remark}[theorem]{Remark}
\newtheorem{defi}[theorem]{Definition}

\newcommand{\R}{\mathbb{R}}
\def\e{\varepsilon }

\begin{document}

\title[Measure-valued free-energy minimizers for trapped bosons]
{Measure-valued free-energy minimizers for trapped bosons with repulsive Coulomb interaction}

\author{Gi-Chan Bae}
\address{School of Mathematical Sciences, Dalian University of Technology, Dalian 116024, P. R. China}
\email{gcbae02@gmail.com}
\author{Jinmyoung Seok}
\address{Department of Mathematics Education, Seoul National University, Seoul 08826, Korea}
\email{jmseok@snu.ac.kr}
\begin{abstract}
We study a free-energy minimization problem for a trapped semiclassical Bose gas with repulsive
Coulomb self-interaction. Because the bosonic entropy density has zero recession slope, the natural
relaxation is posed over nonnegative finite Radon measures on phase space, and a minimizer may have a
singular component. We interpret this component as condensation within the relaxed semiclassical model. 
We prove existence and uniqueness at every positive temperature, derive the Euler--Lagrange contact condition and an obstacle-type formula for the condensed density, 
and establish sharp phase boundaries under only the standing continuity assumptions on the trap. 
At fixed mass there is a unique finite positive critical temperature, with condensation precisely below it. At fixed temperature there is a sharp critical mass, possibly infinite,
separating normal and condensed minimizers. For harmonic traps we obtain a sharp confinement-strength dichotomy. 
We also prove an abstract conditional variational-stability statement for measure-valued curves that conserve mass and satisfy the free-energy inequality.
\end{abstract}

\maketitle

\section*{Introduction}
\addcontentsline{toc}{section}{Introduction}

Bose--Einstein condensation, originating in the work of Bose and Einstein
\cite{Bose1924,Einstein1925}, is one of the classical manifestations of quantum statistics. In the
ideal Bose gas, the Bose--Einstein distribution can accommodate only a finite amount of mass in the
excited states below the critical temperature; the excess mass then appears as a macroscopic
occupation of the ground state. In the relaxed semiclassical model considered here, this
concentration mechanism is represented by allowing the phase-space state to be a measure rather
than only a Lebesgue density. The regular part is interpreted as the thermal cloud and the singular
part as the condensed component of this model. 

The present paper studies this mechanism in a spatially inhomogeneous mean-field model. A bosonic
state is a nonnegative finite Radon measure $F$ on phase space
$\R^3_x\times\R^3_p$. Its spatial density is the marginal $\rho_F$, and the particles interact
through the repulsive Coulomb potential with Newtonian kernel
$V_F=|x|^{-1}*\rho_F$ while being confined by an external potential $\Phi$. At fixed mass $M$ and
temperature $T>0$, the free energy is
\[
 \mathcal F_T(F)=\iint_{\R^6}\left(\frac{|p|^2}{2}+\Phi(x)\right)dF
 +\frac1{8\pi}\int_{\R^3}|\nabla V_F|^2\,dx
 +T\iint_{\R^6}\bigl(F^a\log F^a-(1+F^a)\log(1+F^a)\bigr)\,dpdx.
\]
Here $F^a$ denotes the density of the absolutely continuous part of $F$. The singular part does not contribute directly to the entropy. This is not an additional modelling convention: it is the recession term associated with the convex integrand
\[
 j(s)=s\log s-(1+s)\log(1+s),\qquad j_\infty=\lim_{s\to\infty}\frac{j(s)}s=0.
\]
Thus the correct lower-semicontinuous extension of the entropy functional to measures is precisely the one used above; see \cite{BenAbdallahGambaToscani2011,DemengelTemam} and Lemma~\ref{lsc} for a rigorous justification.

The combination of an external trap and repulsive long-range interaction also has a clear physical motivation. The theory of dilute Bose gases in traps is by now classical; see, for instance, the
review of Dalfovo, Giorgini, Pitaevskii and Stringari
\cite{DalfovoGiorginiPitaevskiiStringari1999} and the positive-temperature analysis of trapped
dilute gases by Deuchert, Seiringer and Yngvason \cite{DeuchertSeiringerYngvason2018}. More
directly related to the Coulombic part of our model are numerical studies of confined charged
bosons. Gonzalez, Partoens, Matulis and Peeters \cite{GonzalezPartoensMatulisPeeters1999}
considered a finite number of charged bosons in a two-dimensional parabolic trap interacting
through the repulsive Coulomb potential $1/r$ and computed ground-state energies by Pad\'e
approximants and variational Monte Carlo. Baksmaty, Yannouleas and Landman
\cite{BaksmatyYannouleasLandman2007} studied rapidly rotating harmonically trapped bosons with
either contact or long-range Coulomb repulsion by exact diagonalization, observing localized
rotating boson-molecule structures. These works are not finite-temperature measure-valued
variational problems, but they illustrate the same confinement--repulsion competition that
underlies our self-consistent free energy.

Literal Coulomb repulsion between neutral ultracold atoms is not the standard experimental platform.
Long-range interacting Bose gases nevertheless play a central role in current experiments and
numerics. The chromium and erbium condensates
\cite{GriesmaierWernerHenslerStuhlerPfau2005,AikawaFrischMarkBaierRietzlerGrimmFerlaino2012}
opened the study of strongly dipolar quantum gases. In such systems the trap geometry and long-range
dipolar interaction can create roton modes, finite-wavelength instabilities, droplets, and spatially
ordered structures; see, for example,
\cite{KadauSchmittWenzelWinkMaierFerrierBarbutPfau2016,ChomazBaierPetterMarkWaechtlerSantosFerlaino2016,ChomazVanBijnenPetterFaraoniBaierBecherMarkWaechtlerSantosFerlaino2018}.
Rydberg-dressed condensates provide another route to effective nonlocal, approximately isotropic
repulsion \cite{HenkelNathPohl2010}, and recent molecular condensates extend the experimental range
of tunable dipolar quantum matter \cite{BigagliYuanZhangBulatovicKarmanStevensonWill2024}. Our
Coulomb model should be understood as a simplified semiclassical mean-field model in this broader
class of trapped Bose systems with nonlocal repulsion.

We next explain the connection with quantum mean-field Bose gas theory, but only at the level at which the analogy is mathematically justified. At the full many-body level, an $N$-body bosonic mixed state $\Gamma_N$ is an operator on the symmetric space $\mathfrak H^N_{\rm sym}$, and the positive-temperature quantum variational principle minimizes
\[
  \operatorname{Tr}_{\mathfrak H^N_{\rm sym}}(H_N\Gamma_N)
  +T\operatorname{Tr}_{\mathfrak H^N_{\rm sym}}(\Gamma_N\log\Gamma_N).
\]
The von Neumann entropy term itself has the same formal expression for bosons, fermions, and distinguishable particles; the bosonic character at this level is encoded in the state space and in the admissible class of density matrices. The scalar Bose entropy used in this paper appears only after one passes to the occupation-number description of bosonic quasi-free states. More precisely, if a gauge-invariant bosonic quasi-free state is parametrized by its one-particle density matrix $\gamma\ge0$, then its Fock-space von Neumann entropy is
\[
 S_{\rm qf}(\gamma)
 =\operatorname{Tr}\bigl((1+\gamma)\log(1+\gamma)-\gamma\log\gamma\bigr)
 =\operatorname{Tr}h(\gamma),
 \qquad h(s)=(1+s)\log(1+s)-s\log s .
\]
Thus $h$ is the occupation-number entropy of a bosonic mode. The entropy in the present paper should be understood as the semiclassical phase-space analogue of this quasi-free entropy: formally, if $\gamma$ has semiclassical symbol $f(x,p)$, then
\[
  \operatorname{Tr}h(\gamma)
  \approx (2\pi\hbar)^{-3}\iint h(f(x,p))\,dxdp .
\]
This is the precise sense in which our entropy is related to the positive-temperature quantum mean-field picture. 

A corresponding operator-level direct-Hartree free energy would have the schematic form
\[
 \operatorname{Tr}\bigl((-\tfrac12\Delta+\Phi)\gamma\bigr)
 +\frac12\iint_{\R^3\times\R^3}\frac{\rho_\gamma(x)\rho_\gamma(y)}{|x-y|}\,dxdy
 -T\operatorname{Tr}h(\gamma),
 \qquad \operatorname{Tr}\gamma=M,
\]
with $\rho_\gamma(x)=\gamma(x,x)$. Its formal Euler--Lagrange equation is
\[
 \gamma=\frac1{\exp((-\tfrac12\Delta+\Phi+V_\gamma-\mu)/T)-1},
\]
and the semiclassical symbol of this equation is precisely the Bose--Einstein formula appearing below for $F_0^a$. We do not derive our model from this operator problem and we do not include exchange or pairing terms. Rather, our functional should be viewed as a semiclassical, direct-Hartree, phase-space free-energy model with repulsive Coulomb self-interaction.

This distinction is important when comparing with the rigorous quantum literature. At zero temperature, Hartree and Gross--Pitaevskii functionals arise as effective variational models for large bosonic many-body systems; see the monograph of Lieb, Seiringer, Solovej and Yngvason \cite{LiebSeiringerSolovejYngvason2005} and the derivation of Hartree theory for generic mean-field Bose systems by Lewin, Nam and Rougerie \cite{LewinNamRougerie2014Hartree}. At positive temperature, nonlinear Gibbs measures have been derived from bosonic grand-canonical quantum Gibbs states in mean-field or high-temperature regimes by Lewin, Nam and Rougerie \cite{LewinNamRougerie2015Gibbs,LewinNamRougerie2020ClassicalField}, and by Fr\"ohlich, Knowles, Schlein and Sohinger \cite{FrohlichKnowlesSchleinSohinger2016}. These works concern quantum Gibbs states and nonlinear field measures, not the measure-valued phase-space minimization problem considered here. Closer in entropy structure are quasi-free and Bogoliubov free-energy theories for positive-temperature Bose gases, where bosonic occupation-number entropies and condensate/thermal decompositions appear explicitly; see, for example, Napi\'orkowski, Reuvers and Solovej \cite{NapiorkowskiReuversSolovej2015BogoliubovI,NapiorkowskiReuversSolovej2015BogoliubovII} and Bach, Breteaux, Chen, Fr\"ohlich and Sigal \cite{BachBreteauxChenFrohlichSigal2016}. Our setting is different because it is trapped, Coulombic, semiclassical, and relaxed to Radon measures, but the entropy mechanism is the same one-particle bosonic mechanism.

A particularly relevant comparison is the work of Deuchert and Seiringer
\cite{DeuchertSeiringer2021}. They derive a semiclassical Bose free-energy
functional from positive-temperature Hartree theory for a harmonically trapped
gas with regular, weak repulsive interactions, and prove that operator-level
Bose--Einstein condensation is equivalent, in their regime, to condensation in
the semiclassical minimizer. In that semiclassical functional the condensate is
represented by a scalar mass at the trap center. By contrast, our admissible
class allows an arbitrary singular phase-space measure, so its spatial support
is selected by the Coulomb contact condition; moreover, the Coulomb kernel lies
outside the regularity class treated in \cite{DeuchertSeiringer2021}. Deuchert
and Seiringer also compare their result with the earlier finite-temperature
Thomas--Fermi analysis of Coulomb-interacting bosons by Narnhofer and Thirring
\cite{NarnhoferThirring1981}, which establishes leading thermodynamic
asymptotics and equivalence of ensembles. Our purpose is different and
complementary: we take the limiting measure-valued free energy as the
variational model and analyze its minimizer directly, obtaining uniqueness,
quasi-everywhere Euler--Lagrange conditions, obstacle regularity, and mass- and
temperature-dependent sharp condensation transitions.

From this perspective the decomposition
\[
 F_0=F_0^a\,dxdp+F_0^s
\]
plays the role of a semiclassical thermal-cloud/condensate decomposition. The absolutely continuous part is the Bose thermal cloud in the self-consistent potential $\Phi+V_{F_0}$, whereas the singular part is the semiclassical counterpart of a macroscopically occupied mode. Because the entropy $h$ grows only sublinearly at large occupation numbers, a singular condensate can carry mass without contributing directly to the entropy in the lower-semicontinuous relaxation. A distinctive feature of the present model is that the condensate is not described by a Gross--Pitaevskii wave function. Instead its spatial distribution is selected by the self-consistent contact condition
\[
 \Phi+V_{F_0}=-c
\]
and, under local $C^2$ regularity of the trap, by an obstacle-type formula for the condensed density. Thus the paper should be understood as a semiclassical Bose free-energy theory in which the usual operator or wave-function description of the condensate is replaced by a measure-valued contact problem for the effective Coulomb potential.

A central mathematical difficulty is that the Bose entropy is not superlinear. For Maxwell--Boltzmann or Fermi--Dirac statistics, entropy often provides strong compactness in $L^1$; for bosons, the sublinear growth at infinity permits concentration. The rigorous treatment of such concentration phenomena was already emphasized in the spatially homogeneous theory. In particular, Ben Abdallah, Gamba and Toscani \cite{BenAbdallahGambaToscani2011} studied minimization problems for sublinear convex entropy functionals arising from nonlinear Fokker--Planck equations with superlinear drift. Their main Bose--Einstein example is generated by the drift $D(f)=f(1+f)$ and leads to the homogeneous Bose--Einstein equilibrium
\[
 f_{\lambda,\theta}(v)=\frac1{\exp(|v|^2/(2\theta)+\lambda)-1}.
\]
The key point in their analysis is that the entropy functional is not coercive enough to prevent weak-* concentration in the space of measures. They therefore minimize the lower-semicontinuous extension of the entropy over nonnegative momentum measures. Under a mass constraint, the family of regular equilibria has a maximal mass $M(0)$; if the prescribed mass exceeds this value, the unique minimizer is the critical regular Bose distribution plus the excess mass concentrated as a Dirac measure at $v=0$. Under simultaneous mass and energy constraints, they identify a scaling-invariant critical ratio separating the purely regular regime from the regime in which a condensate at zero velocity appears.

Their work is an important homogeneous precursor to the present paper. Our variational problem has
the same sublinear-entropic origin, but the spatially inhomogeneous and self-consistent setting
introduces new features. The phase-space measure now depends on both position and momentum, the
spatial marginal creates a Poisson field, and the external potential competes with the repulsive
interaction. Consequently, the singular part is still forced to occur at zero momentum, but its
spatial component is no longer prescribed a priori. Instead it is determined by a contact condition
for the effective potential and, under additional regularity of the trap, by an obstacle-type free
boundary problem.

The main contributions are therefore threefold. 
First, we construct the lower-semicontinuous
measure relaxation together with its self-consistent Coulomb energy and prove that it has a unique
minimizer. We then identify the singular component through a quasi-everywhere contact condition
and, for locally $C^2$ traps, through an obstacle problem that yields a density formula and local
optimal regularity. 
Second, we prove that the normal phase is monotone both in temperature and in mass. The
argument combines a grand-canonical obstacle formulation with a measure-valued Kato inequality and
requires no differentiability of the trap beyond the standing continuity assumption. Consequently,
at every fixed mass there is a unique sharp critical temperature, while at every fixed temperature
there is a sharp critical mass, possibly infinite. Third, for harmonic traps we determine the exact
confinement-strength alternative between a finite critical mass and absence of condensation at every
mass.

The passage from densities to measures is based on the general theory of convex functions of measures developed by Demengel and Temam \cite{DemengelTemam}. In that framework, an integral functional with convex integrand is extended to Radon measures by adding the recession-function contribution on the singular part. 
For the Bose entropy considered here, $j_\infty=0$, so the recession term vanishes.
Thus the singular component carries no direct entropy cost. This is precisely the mechanism that makes a measure-valued formulation unavoidable rather than merely convenient. At the same time, the Coulomb energy imposes potential-theoretic restrictions: finite-energy spatial measures do not charge sets of Newtonian capacity zero. For this reason quasi-continuous representatives and quasi-everywhere statements are built into the Euler--Lagrange formulation.

The first group of results establishes the variational foundation. We prove that the free-energy minimization problem admits a unique minimizer in the measure class. 
In the simpler case without repulsive interactions, whenever condensation occurs, the condensate is concentrated at the minima of the confining potential, i.e., $\operatorname{supp}F_0^s\subset\{(x,0):\Phi(x)=\Phi_{\min}\}$.
In contrast, the inclusion of repulsive self-interaction leads to a substantially richer spatial structure of the condensate. In this case, the minimizer still satisfies a Bose--Einstein formula on its absolutely continuous part,
\[
 F_0^a(x,p)=\frac1{\exp\left(\frac1T(\frac{|p|^2}{2}+\Phi(x)+\widetilde V_{F_0}(x)+c)\right)-1},
\]
where $\widetilde V_{F_0}$ is the quasi-continuous representative of the potential (see Definition~\ref{def-qe}). The singular part is concentrated at zero momentum and lies on the contact set
\[
 \Phi+\widetilde V_{F_0}=-c.
\]
Under the additional regularity assumption $\Phi\in C^2_{\rm loc}$, the contact problem becomes an
obstacle problem: $u_0=\Phi+V_{F_0}+c$ is locally $C^{1,1}$, and the condensed spatial measure is
the restriction to the contact set $\mathcal C=\{u_0=0\}$ of a continuous density. Moreover,
$\mathcal C$ has nonempty interior, and on this open contact region one has
\[
 \rho_{F_0}=\frac1{4\pi}\Delta\Phi.
\]
This gives a concrete spatial description of the condensate.
See Theorem~\ref{prop-EL-minimizer}. Determining the geometry of the
contact set itself, however, requires a further analysis of the contact
condition $u_0=0$. Although a general analysis is beyond the scope of the present work, for some standard radial confining potentials the geometry of the condensate can be characterized. In particular, for
$\Phi(x)=\frac{\omega^2}{2}|x|^2$, the condensate is supported on a ball
$\{|x|\leq R_*\}$, whereas for $\Phi(x)=|x|^4$, it is supported on a
spherical shell $\{R_1\leq |x|\leq R_2\}$. See
Figure~\ref{fig} and Lemmas~\ref{L.A1}--\ref{L.A2} for details.
  The shell-shaped condensate in Figure~\ref{fig}(B) arises from a different mechanism from experimentally realized shell-shaped Bose
gases based on engineered bubble traps \cite{Carollo-Aveline-Rhyno-2022} or Na--Rb interspecies repulsion \cite{Jia-Huang-Qiu-2022}: in our model, it emerges from long-range repulsive self-interaction within a single Bose gas.

\begin{figure}[htbp]
	\centering
	\begin{subfigure}[b]{0.49\textwidth}
		\centering
		\includegraphics[width=\textwidth]{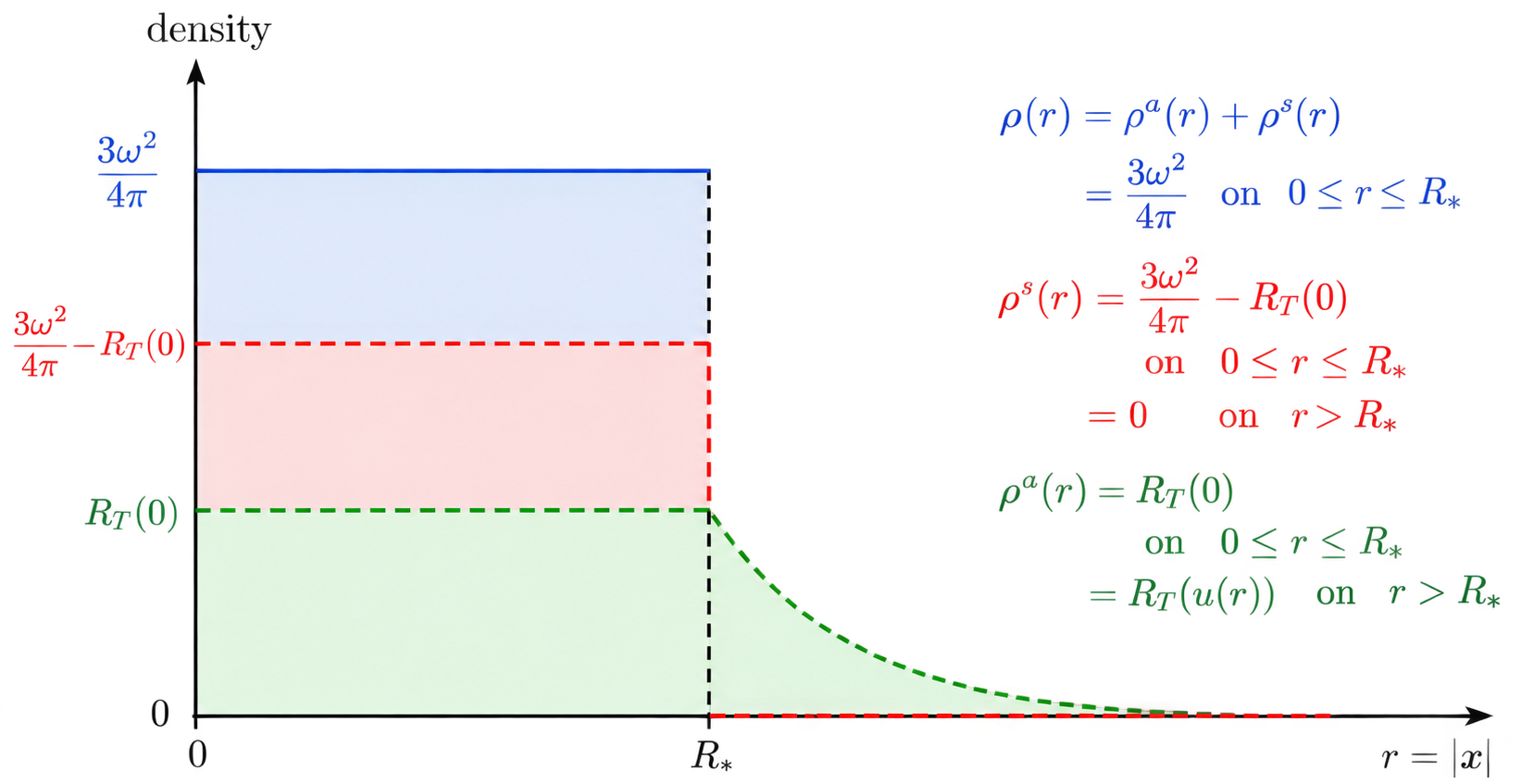}
		\caption{Ball-shaped condensation for $\Phi(x)=\frac{\omega^2}{2}|x|^2$.}
		\label{fig:harmonic}
	\end{subfigure}
	\hfill
	\begin{subfigure}[b]{0.49\textwidth}
		\centering
		\includegraphics[width=\textwidth]{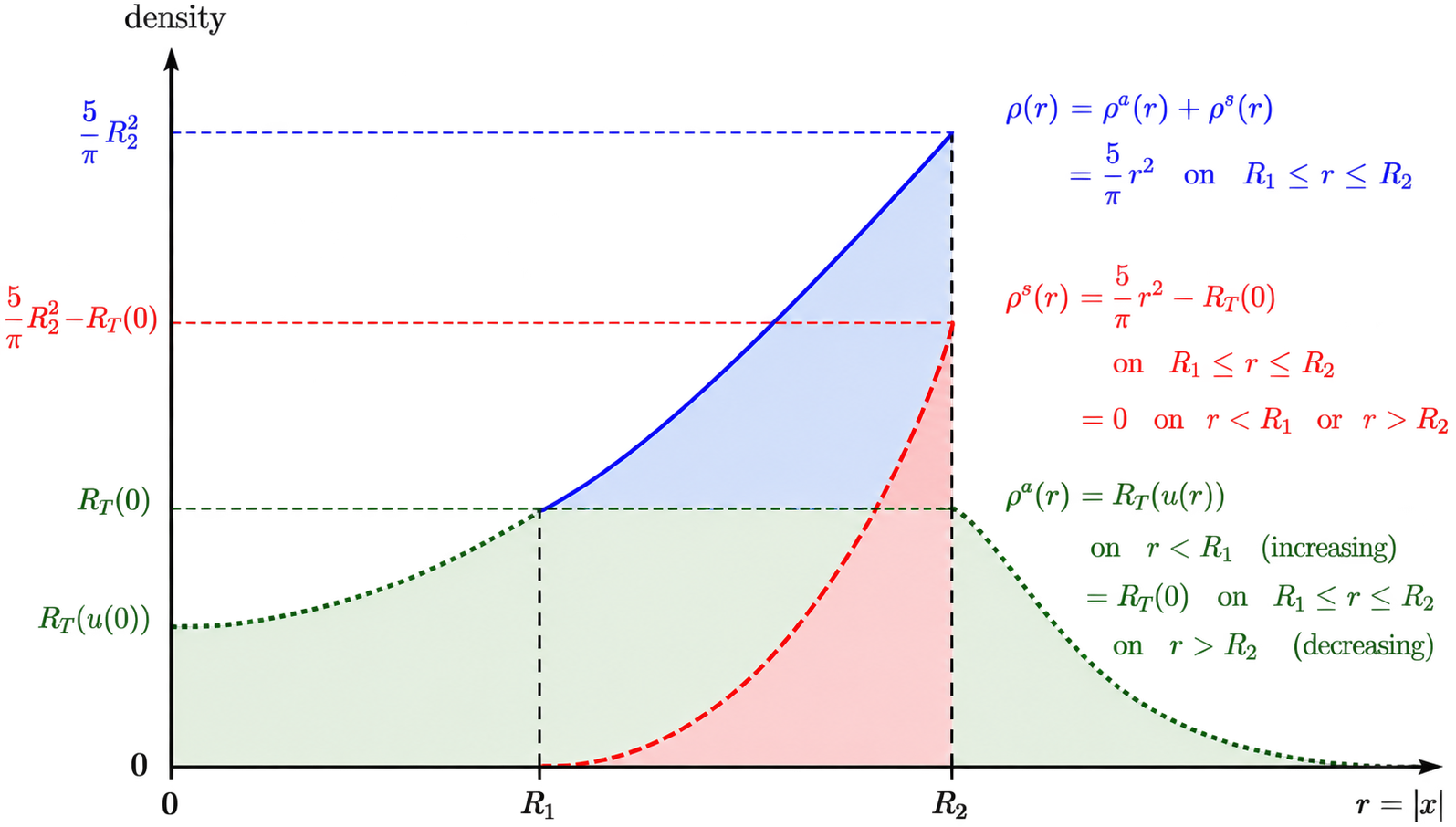}
		\caption{Shell-shaped condensation for $\Phi(x)=|x|^4$.}
		\label{fig:quartic}
	\end{subfigure}
	\caption{Radial density profiles of the thermal cloud and the condensate for two radial confining potentials. In $\R^3$, these profiles correspond to ball-shaped condensation for $\Phi(x)=\frac{\omega^2}{2}|x|^2$ and
		shell-shaped condensation for $\Phi(x)=|x|^4$. See
		Lemmas~\ref{L.A1} and~\ref{L.A2} for details.}
	\label{fig}
\end{figure}

The second group of results concerns transition phenomena. 
 In the absence of repulsive interactions, condensation is characterized by the explicit critical mass
\[
M_c(T)
=
\iint_{\R^6}
\frac{1}
{\exp\left(\frac1T\left(\frac{|p|^2}{2}
	+\Phi(x)-\Phi_{\min}\right)\right)-1}
\,dp\,dx.
\]
At fixed temperature, the system is non-condensed for $M\leq M_c(T)$
and condensed for $M>M_c(T)$; equivalently, at fixed mass there is a
sharp critical temperature separating the two phases. With repulsive
self-interaction the effective potential $\Phi+\widetilde V_{F_0}$ depends self-consistently on the
minimizer, so neither monotonicity is immediate and the critical parameters are implicit. We prove,
nevertheless, that normality at one temperature persists at every higher temperature and that
normality at one mass persists at every lower mass. The comparison is made at fixed chemical
potential through a grand-canonical obstacle problem. A measure-valued Kato inequality rules out
the creation of a new obstacle reaction on the zero set of the difference of two potentials.

At fixed mass, the preliminary high-temperature construction and low-temperature
energy comparison show that both phases occur. The monotonicity and a closedness argument then yield
a unique critical temperature $T_\sharp(M)\in(0,\infty)$: the minimizer is condensed for
$0<T<T_\sharp(M)$ and normal for $T\geq T_\sharp(M)$. Here, a normal state means a purely
absolutely continuous state. The low-temperature input is not based on normal-branch capacity; it
uses the strict gap between the zero-temperature energy minimum and the energy level of all normal
states.

At fixed temperature, the same comparison produces a sharp critical mass
$M_\sharp(T)\in(0,\infty]$. The minimizer is normal for $0<M\leq M_\sharp(T)$ and condensed for
$M>M_\sharp(T)$. Under strong confinement, either superquadratic confinement or a sufficiently
strong harmonic trap, $M_\sharp(T)$ is finite. For a harmonic potential
\[
 \Phi(x)=\Phi(0)+\frac{\omega^2}{2}|x|^2,
\]
there is a threshold governed by
\[
 3\omega^2 \gtrless 4\pi R_T(0),
 \qquad
 R_T(0)=\int_{\R^3}\frac1{\exp(|p|^2/(2T))-1}\,dp.
\]
If $3\omega^2>4\pi R_T(0)$, then $M_\sharp(T)<\infty$. If
$3\omega^2\le4\pi R_T(0)$, then $M_\sharp(T)=\infty$ and no condensation occurs for any $M>0$.

Finally, we prove an abstract conditional variational-stability statement. It applies to any
measure-valued curve for which mass is conserved and the free energy is nonincreasing on the time
interval under consideration. Three kinetic equations are displayed only as formal motivations:
Vlasov--Poisson, the spatially inhomogeneous bosonic Boltzmann--Poisson equation, and a bosonic
Vlasov--Fokker--Planck equation. No existence theorem for measure-valued solutions of those
equations is asserted here. For related homogeneous Bose--Einstein kinetic theory and convergence to
equilibrium, see for example \cite{CaiLu2018}. The proof is the variational Lyapunov argument: the
free-energy gap equals a relative free-energy functional at the unique minimizer, and
minimizing-sequence compactness converts energetic control into weak-* and potential stability.

The paper is organized as follows. Section~1 gives the precise setting and states the main results.
Section~2 proves compactness, lower semicontinuity, existence, uniqueness, and the Euler--Lagrange
characterization. Section~3 proves the preliminary low- and high-temperature bounds,
develops the grand-canonical obstacle comparison, and establishes the sharp temperature transition.
Section~4 identifies the sharp critical mass and proves the harmonic-trap dichotomy. Section~5
proves the abstract conditional stability result.

\section{Setting and main results}

This section fixes the variational framework and states the main results. We first introduce the
admissible measure class and the Coulomb energy, and then recall the potential-theoretic terminology
needed for finite-energy measures. We next define the entropy and free-energy functionals. The main
results stated at the end of this section are proved in the subsequent sections.

\subsection{The admissible measure class and Coulomb energy}
For a real-valued function $\Phi$, write
\[
 \Phi_+:=\max\{\Phi,0\},\qquad
 \Phi_-:=\max\{-\Phi,0\}.
\]
Let $\Phi \in C(\R^3)$ satisfy
\[
\lim_{|x|\to \infty}\Phi(x) = \infty, \qquad
\int_{\R^3}e^{- \frac12\sqrt{\Phi_+(x)}}\,dx < \infty.
\]
Since $\Phi$ is continuous and confining, it attains its minimum. We set
\[
 \Phi_{\min}:=\min_{\R^3}\Phi;
\]
in particular, $\Phi_-$ is bounded. We use
\[
 \dot H^1(\R^3):=\{u\in L^6(\R^3):\nabla u\in L^2(\R^3)\},
 \qquad \|u\|_{\dot H^1}:=\|\nabla u\|_{L^2}.
\]
We also write
\[
 \mathcal M(\R^6):=C_0(\R^6)^*
\]
for the space of finite signed Radon measures on phase space.

For a fixed mass $M>0$, we define the admissible class by
\[
\mathcal{A}_M \coloneqq
\left\{ F \in \mathcal{M}(\R^6) ~\middle|~
\begin{aligned}
&F \geq 0,\quad
\int_{\R^6}\left(\frac{|p|^2}{2}+\Phi_+(x)\right)dF<\infty,\\
&V_F \in \dot{H}^1(\R^3),\quad F(\R^6)=M
\end{aligned}\right\},
\]
and, whenever the mass is fixed, abbreviate $\mathcal A:=\mathcal A_M$. Here
\[
\rho_F(E)\coloneqq F(E\times\R^3),\qquad
V_F(x)=\int_{\R^3}\frac{1}{|x-y|}\,d\rho_F(y).
\]

\begin{prop}\label{prop-coulomb-energy}
Let $\rho$ be a nonnegative finite Radon measure in $\R^3$. Then
\[
V_\rho(x) = \int_{\R^3}\frac{1}{|x-y|}\,d\rho(y) \in \dot{H}^1(\R^3)
\]
if and only if
\[
\iint_{\R^3\times\R^3}\frac{1}{|x-y|}\,d\rho(x)d\rho(y)<\infty.
\]
In this case,
\[
\int_{\R^3}|\nabla V_\rho|^2\,dx
=4\pi\iint_{\R^3\times\R^3}\frac{1}{|x-y|}\,d\rho(x)d\rho(y).
\]
\end{prop}

\subsection{Capacity, quasi-continuity, and finite-energy measures}

The pairing of a Radon measure with a $\dot H^1$ function is not well-defined in general, since $\dot H^1$ functions are defined only up to sets of Lebesgue measure zero. For instance, $\phi$ and $\phi+\mathbf 1_{\{0\}}$ represent the same $\dot H^1$ function but yield different pairings with $\delta_0$. To make this pairing well-defined, we introduce the notion of quasi-continuity.

\begin{defi}[Quasi-continuity]\label{def-qe}
We use the following standard potential-theoretic terminology.
\begin{enumerate}[\rm(i)]
\item For a compact set $K\subset\R^3$, its Newtonian capacity is defined by
\[
 \operatorname{Cap}(K):=
 \inf\left\{\int_{\R^3}|\nabla\varphi|^2\,dx:
 \varphi\in C_c^\infty(\R^3),\ \varphi\ge0,\ \varphi\geq1
 \text{ in a neighborhood of }K\right\}.
\]
For an open set $O\subset\R^3$, we set
\[
 \operatorname{Cap}(O):=
 \sup\{\operatorname{Cap}(K):K\subset O,\ K\text{ compact}\}.
\]
For an arbitrary set $E\subset\R^3$, we use the outer regular definition
\[
 \operatorname{Cap}(E):=
 \inf\{\operatorname{Cap}(O):E\subset O,\ O\subset\R^3\text{ open}\}.
\]
For Borel sets this capacity is capacitable, namely
\[
 \operatorname{Cap}(E)
 =\sup\{\operatorname{Cap}(K):K\subset E,\ K\text{ compact}\}.
\]

\item A property is said to hold quasi-everywhere, abbreviated q.e., if it holds outside a set of
Newtonian capacity zero.

\item A function $u:\R^3\to[-\infty,\infty]$ is called quasi-continuous if, for every
$\varepsilon>0$, there exists an open set $G\subset\R^3$ with
$\operatorname{Cap}(G)<\varepsilon$ such that $u|_{\R^3\setminus G}$ is finite and continuous.

\item For $\phi\in\dot H^1(\R^3)$, a quasi-continuous representative of $\phi$ is a
quasi-continuous function $\widetilde\phi$ such that $\widetilde\phi=\phi$ Lebesgue-a.e.
\end{enumerate}
\end{defi}

It is standard that every $\phi\in\dot H^1(\R^3)$ has a quasi-continuous representative, unique up to a set of capacity zero; see, for instance, \cite[Section 6.1]{EvansGariepy2015} or \cite[Chapter 6]{AdamsHedberg1996}. We shall also use the elementary consequence of Sobolev's inequality that sets of zero Newtonian capacity have zero Lebesgue measure. Indeed, if $K\subset\R^3$ is compact and $\varphi\in C_c^\infty(\R^3)$ satisfies $\varphi\ge0$ and $\varphi\ge1$ in a neighborhood of $K$, then
\[
 |K|\le \int_K |\varphi|^6\,dx\le \|\varphi\|_{L^6}^6
 \le C\left(\int_{\R^3}|\nabla\varphi|^2\,dx\right)^3,
\]
and hence $|K|^{1/3}\le C\operatorname{Cap}(K)$. The assertion for arbitrary Borel sets follows from the outer regularity of capacity and the inner regularity of Lebesgue measure. Hereafter we always use the quasi-continuous representative when a function in $\dot H^1$ is integrated against a finite-energy measure.

\begin{prop}\label{prop-quasicont}
Let $F \in \mathcal{A}$. Then for every quasi-continuous $\phi\in\dot H^1(\R^3)$ the integral
$\int_{\R^3}\phi\,d\rho_F$ is well defined. Moreover,
\[
\int_{\R^3}\nabla V_F\cdot\nabla\phi\,dx
=4\pi\int_{\R^3}\phi\,d\rho_F.
\]
\end{prop}

\subsection{Entropy and free-energy functionals}
For $F\in\mathcal A$, write its Lebesgue decomposition as
\[
F=F^a\,dpdx+F^s,
\]
where
\[
F^a\in L^1(\R^6),\qquad F^s\in\mathcal M(\R^6),
\qquad F^a,F^s\geq0,\qquad F^s\perp dpdx.
\]
We use the convention
\[
h(s)\coloneqq (1+s)\log(1+s)-s\log s,
\qquad h(0)=0.
\]
Then $h'(s)=\log(1+1/s)$ and $h''(s)=-1/(s(1+s))<0$ for $s>0$. The bosonic entropy is
\[
H(F) \coloneqq \iint_{\R^6}\Big(F^a\log F^a-(1+F^a)\log(1+F^a)\Big)\,dpdx
=-\iint_{\R^6}h(F^a)\,dpdx.
\]
We now explain why the singular part does not appear in this definition. Let
$\lambda:=\mathcal L^6$ denote Lebesgue measure on phase space and set
\begin{align}\label{jdef}
 j(s):=-h(s)=s\log s-(1+s)\log(1+s),\qquad s\geq0.
\end{align}
The function $j$ is convex, and its recession function in every nonnegative direction is zero:
\[
 j^\infty(z):=\lim_{t\to\infty}\frac{j(tz)}{t}
 =z\lim_{s\to\infty}\frac{j(s)}s=0,\qquad z\geq0.
\]
For a convex integral functional initially defined on densities, the standard weak-* lower
semicontinuous relaxation to finite measures $F=F^a\lambda+F^s$ is governed by the recession
function and has the form
\[
 \overline H(F)
 =\int_{\R^6}j(F^a)\,d\lambda
 +\int_{\R^6}j^\infty\left(\frac{dF^s}{d|F^s|}\right)d|F^s|;
\]
see, for example, \cite{DemengelTemam}. Since $F^s$ is nonnegative,
$dF^s/d|F^s|=1$ $|F^s|$-a.e., and $j^\infty(1)=0$. Consequently,
\[
 \overline H(F)=\int_{\R^6}j(F^a)\,dpdx=H(F).
\]
Thus the absence of a singular entropy term is not an additional modeling convention: it is forced
by the recession-function formula for the lower-semicontinuous relaxation. Equivalently,
concentration into a singular measure has zero asymptotic entropy contribution per unit concentrated
mass. On compact phase-space sets this is the standard relaxation formula for convex functionals of
measures. Since the present problem is posed on the whole space and $j$ is not bounded below, the
global lower-semicontinuity statement also requires the moment and tightness estimates proved below.

Total energy and free energy: for $F \in \mathcal{A}$,
\[
\mathcal E(F)
= \iint_{\R^6}\left(\frac{|p|^2}{2}+\Phi(x) \right)\, dF
+\frac{1}{8\pi} \int_{\R^3}|\nabla V_F|^2\,dx,
\]
and
\[
\mathcal{F}_T(F)=\mathcal E(F)+T H(F),
\]
where $T>0$ is fixed. When there is no risk of confusion, we simply write $\mathcal F$ for
$\mathcal F_T$.

The associated minimization problem at fixed mass $M$ is
\begin{equation}\label{mp-boson}
\mathcal{F}_{\min}(T) \coloneqq \inf_{F \in \mathcal{A}}\mathcal{F}_T(F).
\end{equation}
The dependence of $\mathcal F_{\min}$ and of $\mathcal A$ on the fixed mass $M$ is suppressed in
the notation.

\subsection{Existence and Euler--Lagrange characterization}
The first main variational statement is the existence and uniqueness of the equilibrium state. The next theorem then identifies the equilibrium through its regular Bose distribution and its singular contact condition.

\begin{prop}\label{prop-existence-minimizer}
Let $T>0$. The minimization problem \eqref{mp-boson} admits a unique minimizer
$F_0\in\mathcal A$.
Moreover, if $\Phi$ is radial, then $F_0$ is radial, i.e. for every $R\in SO(3)$ and every
Borel set $A\times B\subset\R^3\times\R^3$,
\[
F_0(RA\times RB)=F_0(A\times B).
\]
\end{prop}

\begin{theorem}\label{prop-EL-minimizer}
Let $F_0$ be a minimizer of \eqref{mp-boson} with $T>0$, and let
$\widetilde V_{F_0}$ denote the quasi-continuous representative of $V_{F_0}$. Then there exists a
constant $c\in\R$ such that the following assertions hold.
\begin{enumerate}[\rm(i)]
\item If
\[
 u_0(x):=\Phi(x)+\widetilde V_{F_0}(x)+c,
\]
then
\[
 u_0\geq0 \qquad \text{q.e. in }\R^3,
\]
and the absolutely continuous part is given by
\[
F_0^a(x,p)=
\frac{1}{\exp\left(\frac1T\left(\frac{|p|^2}{2}+\Phi(x)+\widetilde V_{F_0}(x)+c\right)\right)-1}
\]
for a.e. $(x,p)\in\R^6$.

\item The singular part is concentrated at zero momentum. More precisely,
\[
F_0^s=\rho_0^s\otimes\delta_{p=0}
\]
for some finite nonnegative Radon measure $\rho_0^s$ on $\R^3$, and

\[
u_0=0,
\quad\text{equivalently,}\quad
\Phi+\widetilde V_{F_0}=-c
\qquad \rho_0^s\text{-a.e.}
\]

In particular, the condensed part is compactly supported.

\item Define
\[
\rho_0^a(x):=\int_{\R^3}F_0^a(x,p)\,dp.
\]
Assume in addition that $\Phi\in C^2_{\rm loc}(\R^3)$ and $F_0^s\neq0$. Then $u_0$ has a
$C^{1,1}_{\rm loc}$ representative, still denoted by $u_0$, and $\rho_0^a$ has a continuous
representative, which we use below. Set
\[
 \mathcal C:=\{x\in\R^3:u_0(x)=0\},
 \qquad
 g_0(x):=\frac{1}{4\pi}\Delta\Phi(x)-\rho_0^a(x).
\]
Then $g_0\in C(\R^3)$, and $\rho_0^s$ is absolutely continuous with respect to the Lebesgue measure, with Radon--Nikodym density
\[
 \frac{d\rho_0^s}{dx}(x)=g_0(x){\bf 1}_{\mathcal C}(x)
 \qquad\text{for a.e. }x\in\R^3.
\]
Moreover,
\[
 \operatorname{int}\mathcal C\neq\varnothing.
\]
Consequently,
\[
 \rho_{F_0}=\frac{1}{4\pi}\Delta\Phi
 \qquad\text{in }\mathcal D'(\operatorname{int}\mathcal C).
\]
\end{enumerate}
\end{theorem}

\begin{remark}\label{rem-qc-and-boundary}
The quasi-continuous formulation in Theorem~\ref{prop-EL-minimizer}(i)--(ii) is essential under
only the standing continuity assumptions on $\Phi$. Although the formula in (i) implies that the
regular spatial density is bounded by the finite number $R_T(0)$ (defined by the momentum integral in Theorem~\ref{thm-M}), the full potential contains the singular contribution
$V_{\rho_0^s}$. A finite-Coulomb-energy measure does not charge sets of Newtonian capacity zero, and
therefore has no atoms, but it may still be singular with respect to Lebesgue measure. Consequently
$V_{\rho_0^s}$ need not have a continuous representative, and $u_0$ should not be claimed to be a
continuous function at this level of generality. The stronger conclusions in (iii), including
$u_0\in C^{1,1}_{\rm loc}$ and the global density representation of $\rho_0^s$, use the additional assumption
$\Phi\in C^2_{\rm loc}(\R^3)$ and the associated obstacle regularity theory.
\end{remark}

\begin{remark}
Theorem~\ref{prop-EL-minimizer} determines the density of the
condensed part explicitly, but does not characterize the geometry of
the contact set $\mathcal C$. To determine $\mathcal C$, one needs to
further analyze the contact condition $u_0=0$. When the confining
potential is radial, the minimizer $F_0$ and the potential $V_{F_0}$
are also radial, reducing this analysis to a one-dimensional problem.
For the standard examples,
$\Phi(x)=\Phi(0)+\frac{\omega^2}{2}|x|^2$ yields a ball-shaped
condensate, whereas $\Phi(x)=|x|^4$ yields a shell-shaped condensate.
We refer to Lemmas~\ref{L.A1} and~\ref{L.A2} for details.
\end{remark}

\subsection{Sharp critical temperature at fixed mass}
We next state the sharp temperature transition. Its proof, given in
Section~\ref{sec-temp-transition}, combines preliminary low- and high-temperature estimates with a
comparison principle for the associated grand-canonical obstacle problem.

\begin{theorem}\label{thm-T}
Fix $M>0$. For each $T>0$, let $F_T$ be the minimizer of \eqref{mp-boson}, and denote by
$\rho_T^s$ the spatial projection of its singular part. Then there exists a unique critical
temperature
\[
 T_\sharp(M)\in(0,\infty)
\]
such that
\[
 \rho_T^s\neq0 \quad\text{for }0<T<T_\sharp(M),
 \qquad
 \rho_T^s=0 \quad\text{for }T\geq T_\sharp(M).
\]
\end{theorem}

\subsection{Sharp critical mass and harmonic-trap dichotomy}
At fixed temperature the normal phase is a closed initial interval in the prescribed mass. The
following theorem identifies its endpoint and records the harmonic-trap dichotomy, where the
strength of the confinement determines whether the normal branch can accommodate arbitrary mass.

\begin{theorem}\label{thm-M}
Fix $T>0$, and let $F_0$ be the minimizer of \eqref{mp-boson} with mass $M>0$. Set
\[
R_T(0):=\int_{\R^3}\frac{1}{\exp\left(\frac{|p|^2}{2T}\right)-1}\,dp
=(2\pi T)^{3/2}\zeta(3/2),
\]
where $\zeta(s)=\sum_{n=1}^\infty n^{-s}$ denotes the Riemann zeta function. The identity follows from spherical coordinates and the standard formula
\[
 \int_0^\infty \frac{s^{a-1}}{e^s-1}\,ds=\Gamma(a)\zeta(a),\qquad a>1,
\]
applied with $a=3/2$. Then the following assertions hold.
\begin{enumerate}[(a)]
\item \emph{Sharp mass transition.} There exists a uniquely determined extended critical mass
\[
 M_\sharp(T)\in(0,\infty]
\]
such that
\[
 \rho_0^s=0 \quad\Longleftrightarrow\quad 0<M\leq M_\sharp(T),
 \qquad
 \rho_0^s\neq0 \quad\Longleftrightarrow\quad M>M_\sharp(T).
\]

\item \emph{Finiteness under strong confinement.} If either
\[
\Phi(x)\geq a|x|^q-b\qquad\text{for large }|x|\text{ with }q>2,
\]
or
\[
\Phi(x)=\Phi(0)+\frac{\omega^2}{2}|x|^2
\qquad\text{and}\qquad
3\omega^2>4\pi R_T(0).
\]
then $M_\sharp(T)<\infty$.

\item \emph{No condensation under weak harmonic trapping.} Assume
\[
\Phi(x)=\Phi(0)+\frac{\omega^2}{2}|x|^2
\qquad\text{and}\qquad
3\omega^2\leq4\pi R_T(0).
\]
Then $M_\sharp(T)=\infty$; equivalently, the minimizer is purely absolutely continuous for every
$M>0$.

\item \emph{Temperature--mass correspondence.} The map $T\mapsto M_\sharp(T)$ is nondecreasing,
and
\[
 T_\sharp(M)=\min\{T>0:M\leq M_\sharp(T)\}.
\]
\end{enumerate}
\end{theorem}

\subsection{Abstract conditional variational stability}
Let $F_0$ be the unique minimizer of \eqref{mp-boson} at temperature $T>0$ and mass $M$,
and let $c$ and $u_0$ be as in Theorem~\ref{prop-EL-minimizer}. 
With $j$ defined in \eqref{jdef}, we define the relative entropy integrand
\[
 \mathcal J(a\mid b):=j(a)-j(b)-j'(b)(a-b),\qquad a\geq0,\ b>0.
\]
For $F\in\mathcal A$ with the same mass $M$, we define the relative free-energy functional at
$F_0$ by
\begin{align}\label{rel-distance-main}
\mathcal D_T(F\mid F_0)
&:=T\iint_{\R^6}\mathcal J(F^a\mid F_0^a)\,dpdx
   +\iint_{\R^6}\left(\frac{|p|^2}{2}+u_0(x)\right)\,dF^s(x,p)  \notag\\
&\quad +\frac1{8\pi}\int_{\R^3}|\nabla(V_F-V_{F_0})|^2\,dx.
\end{align}
Here $u_0$ is understood through its quasi-continuous representative. The second term is well-defined
because finite-Coulomb-energy spatial measures do not charge sets of Newtonian capacity zero.

We shall use the following abstract class of curves. A curve
$F:[0,\tau]\to\mathcal A$ is called an admissible free-energy curve on $[0,\tau]$ if it is
narrowly continuous as a curve of finite Radon measures, has fixed mass $M$, and satisfies the
free-energy inequality
\[
 \mathcal F_T(F(t))\leq \mathcal F_T(F(0)),\qquad 0\leq t\leq\tau.
\]
This definition is deliberately conditional: it applies when a solution concept for a chosen
kinetic equation produces such a curve, but no well-posedness assertion is included in the result below.

\begin{prop}[Conditional variational stability]\label{thm-dyn-stab}
Let $F_0$ be the unique minimizer of \eqref{mp-boson} at temperature $T>0$ and mass $M$. Let
$F:[0,\tau]\to\mathcal A$ be an admissible free-energy curve with the same mass. Then
\[
 \mathcal D_T(F(t)\mid F_0)\leq \mathcal D_T(F(0)\mid F_0),
 \qquad 0\leq t\leq\tau.
\]
In particular, $F_0$ is conditionally stable with respect to the relative free-energy functional.
Moreover, this
energetic stability implies variational weak stability: if $\mathcal U$ is any weak-* neighborhood of
$F_0$ in $\mathcal M(\R^6)$ and $\eta>0$, then there exists $\delta>0$ such that
\[
 \mathcal D_T(F(0)\mid F_0)<\delta
\]
implies, for all $0\leq t\leq\tau$,
\[
 F(t)\in\mathcal U,
 \qquad
 \|\nabla(V_{F(t)}-V_{F_0})\|_{L^2(\R^3)}<\eta.
\]
\end{prop}

\begin{remark}
Formal model equations that motivate Proposition~\ref{thm-dyn-stab} include the Vlasov--Poisson flow,
the spatially inhomogeneous bosonic Boltzmann--Poisson equation, and bosonic
Vlasov--Fokker--Planck dynamics. The proposition applies only after an independently chosen solution
theory supplies the required mass conservation and free-energy inequality. The equations are
displayed in Section~\ref{sec-dyn-stab}.
\end{remark}

\section{Variational compactness and Euler--Lagrange equations}

The purpose of this section is to justify the measure-valued variational framework used in Section~1. We first record the Coulomb energy facts and quasi-continuity properties needed to pair potentials with finite-energy measures. We then prove lower semicontinuity, existence and uniqueness of minimizers, and finally the Euler--Lagrange and obstacle characterizations.

\subsection{Coulomb energy and quasi-continuity}
\begin{proof}[Proof of Proposition \ref{prop-coulomb-energy}]
We use the Fourier transform convention
\[
 \widehat \mu(\xi)=\int_{\R^3}e^{-ix\cdot\xi} d\mu,
\]
for every finite Radon measure $\mu$ on $\R^3$
The standard Fourier representation of Newtonian energy states that, for a finite positive measure,
\begin{equation}\label{eq-fourier-coulomb}
 I(\rho):=\iint_{\R^3\times\R^3}\frac{d\rho(x)d\rho(y)}{|x-y|}
 =\frac{1}{(2\pi)^3}\int_{\R^3}\frac{4\pi}{|\xi|^2}
 |\widehat\rho(\xi)|^2\,d\xi,
\end{equation}
with both sides allowed to be infinite. It follows first for smooth compactly supported densities
from $\widehat{|x|^{-1}}=4\pi|\xi|^{-2}$ and then for positive measures by regularization and
truncation of the Newtonian kernel.

Suppose that $I(\rho)<\infty$. Choose a standard mollifier $\eta_{\e}(x) = \e^{-3}\eta(x/\e)$ with
$|\widehat\eta|\leq1$, and set
\[
 \rho_\varepsilon=\eta_\varepsilon*\rho,\qquad
 V_\varepsilon=|x|^{-1}*\rho_\varepsilon.
\]
For $\varepsilon,\delta>0$, Plancherel's theorem and \eqref{eq-fourier-coulomb}, applied also to the
signed measure $\rho_\varepsilon-\rho_\delta$, give
\[
 \|\nabla(V_\varepsilon-V_\delta)\|_2^2
 =\frac{16\pi^2}{(2\pi)^3}\int_{\R^3}
 \frac{|\widehat\rho(\xi)|^2
 |\widehat\eta(\varepsilon\xi)-\widehat\eta(\delta\xi)|^2}{|\xi|^2}\,d\xi.
\]
The integrand tends pointwise to zero and is dominated by an integrable multiple of the integrand
in \eqref{eq-fourier-coulomb}. Thus $\{V_\varepsilon\}$ is Cauchy in $\dot H^1$. Since
$V_\varepsilon=\eta_\varepsilon*V_\rho$ and the Newtonian potential of a finite measure is locally
integrable, $V_\varepsilon\to V_\rho$ in distributions. Hence $V_\rho\in\dot H^1$. A second
application of dominated convergence yields
\[
 \|\nabla V_\rho\|_2^2
 =\frac{16\pi^2}{(2\pi)^3}\int_{\R^3}
 \frac{|\widehat\rho(\xi)|^2}{|\xi|^2}\,d\xi
 =4\pi I(\rho).
\]

Conversely, suppose that $V_\rho\in\dot H^1$. The distributional identity
$-\Delta V_\rho=4\pi\rho$ implies, in Fourier variables,
\[
 |\xi|^2\widehat V_\rho=4\pi\widehat\rho.
\]
The Fourier characterization of $\dot H^1$ therefore gives
\[
 \frac{1}{(2\pi)^3}\int_{\R^3}\frac{4\pi}{|\xi|^2}
 |\widehat\rho(\xi)|^2\,d\xi
 =\frac1{4\pi}\|\nabla V_\rho\|_2^2<\infty.
\]
Formula \eqref{eq-fourier-coulomb} now proves $I(\rho)<\infty$ and the asserted identity.
\end{proof}

\begin{proof}[Proof of Proposition \ref{prop-quasicont}]
Set $\mu=\rho_F$. Since $V_F\in\dot H^1$, Proposition \ref{prop-coulomb-energy} implies that
$\mu$ has finite Coulomb energy. We first prove explicitly that $\mu$ charges no set of zero
Newtonian capacity. For $\phi\in C_c^\infty(\R^3)$, Fubini's theorem and
$-\Delta |x|^{-1}=4\pi\delta_0$ give
\[
\int_{\R^3}\nabla V_F\cdot\nabla\phi\,dx
=\int_{\R^3}V_F(-\Delta\phi)\,dx
=4\pi\int_{\R^3}\phi\,d\mu.
\]
Let $K\subset\R^3$ be compact. If $\varphi\in C_c^\infty(\R^3)$, $\varphi\ge0$, and $\varphi\ge1$ in a neighborhood
of $K$, then
\[
 \mu(K)\le \int_{\R^3}\varphi\,d\mu
 =\frac1{4\pi}\int_{\R^3}\nabla\varphi\cdot\nabla V_F\,dx
 \le \frac1{4\pi}\|\nabla\varphi\|_2\|\nabla V_F\|_2.
\]
Taking the infimum over such $\varphi$ and using Proposition \ref{prop-coulomb-energy}, we obtain
the energy-capacity estimate
\[
 \mu(K)^2
 \le \frac1{4\pi}
 \left(\iint_{\R^3\times\R^3}\frac{1}{|x-y|}\,d\mu(x)d\mu(y)\right)
 \operatorname{Cap}(K).
\]
Equivalently, up to the harmless normalization constant in the definition of capacity,
\[
 \mu(K)\le C I(\mu)^{1/2}\operatorname{Cap}(K)^{1/2},
 \qquad
 I(\mu):=\iint_{\R^3\times\R^3}\frac{1}{|x-y|}\,d\mu(x)d\mu(y).
\]
Now let $E$ be a Borel set with $\operatorname{Cap}(E)=0$. By the outer regular definition of
capacity, for every $\varepsilon>0$ there is an open set $O\supset E$ with
$\operatorname{Cap}(O)<\varepsilon$. For every compact $K\subset O$, the preceding estimate gives
\[
 \mu(K)\le C I(\mu)^{1/2}\operatorname{Cap}(O)^{1/2}
 \le C I(\mu)^{1/2}\varepsilon^{1/2}.
\]
Taking the supremum over compact $K\subset O$ and using the inner regularity of the Radon measure
$\mu$, we get the same bound for $\mu(O)$. Hence
\[
 \mu(E)\le \mu(O)\le C I(\mu)^{1/2}\varepsilon^{1/2}.
\]
Letting $\varepsilon\downarrow0$, we conclude that $\mu(E)=0$. Thus $\rho_F$ charges no sets of
zero Newtonian capacity.

Consequently, two quasi-continuous representatives of the same
$\dot H^1$ function agree $\rho_F$-a.e., so the integral against
$\rho_F$ is independent of the representative. For a general
$\phi\in\dot H^1(\R^3)$, choose $\phi_n\in C_c^\infty(\R^3)$ such that
$\phi_n\to\phi$ in $\dot H^1$. For $\psi\in C_c^\infty(\R^3)$, we have
\[
\left|\int_{\R^3}\psi\,d\rho_F\right|
=
\frac{1}{4\pi}
\left|\int_{\R^3}\nabla V_F\cdot\nabla\psi\,dx\right|
\leq
\frac{1}{4\pi}\|\nabla V_F\|_2\|\nabla\psi\|_2.
\]
Moreover, approximating $|\psi|$ in $\dot H^1$ by nonnegative functions
in $C_c^\infty(\R^3)$ and using Fatou's lemma gives
\[
\int_{\R^3}|\psi|\,d\rho_F
\leq
\frac{1}{4\pi}\|\nabla V_F\|_2\|\nabla\psi\|_2.
\]
Applying this estimate to $\phi_n-\phi_k$ shows that $\{\phi_n\}$ is Cauchy in $L^1(d\rho_F)$. Hence $\phi_n\to \phi_*$ in $L^1(d\rho_F)$ for some $\phi_*\in L^1(d\rho_F)$. On the other hand, by the standard
quasi-continuity theorem, a subsequence converges quasi-everywhere to
the quasi-continuous representative $\widetilde\phi$. Since $\rho_F$
does not charge sets of zero capacity, the two limits agree
$\rho_F$-a.e. Therefore,
\[
\int_{\R^3}|\widetilde\phi|\,d\rho_F
\leq
\frac{1}{4\pi}\|\nabla V_F\|_2\|\nabla\phi\|_2,
\]
and $\phi_n\to\widetilde\phi$ in $L^1(d\rho_F)$.
Thus the continuous
extension of the smooth pairing agrees with integration against
$\widetilde\phi$. Passing to the limit in the smooth identity yields
\[
\int_{\R^3}\nabla V_F\cdot\nabla\phi\,dx
=
4\pi\int_{\R^3}\widetilde\phi\,d\rho_F.
\]
\end{proof}

\subsection{Entropy estimates and lower semicontinuity}
\begin{lemma}\label{lem-decay}
For all $s\geq0$ and $a\geq0$,
\[
h(s)\leq (2+a)s+2e^{-a/2}.
\]
Consequently, for any density $F(x,p)\geq0$,
\begin{equation}\label{decay}
h(F(x,p)) \leq (2+|p|+\sqrt{\Phi_+(x)})F(x,p)
+2e^{-\frac{|p|+\sqrt{\Phi_+(x)}}{2}}.
\end{equation}
\end{lemma}
\begin{proof}
For $b>0$, the concavity of $h$ and the identity $h'(s)=\log(1+1/s)$ give
\[
\sup_{s\geq0}\{h(s)-bs\}=-\log(1-e^{-b}).
\]
Taking $b=2+a$, we get
\[
h(s)\leq (2+a)s-\log(1-e^{-(2+a)}).
\]
Since $-\log(1-r)\leq 2r$ for $0\leq r\leq e^{-2}$ and
$e^{-(2+a)}\leq e^{-a/2}$ for $a\geq0$, the first estimate follows. The second one is obtained by
choosing $a=|p|+\sqrt{\Phi_+(x)}$.
\end{proof}

\begin{lemma}\label{lsc}
Let $\{F_n\}\subset\mathcal A$ and let $F_0\in\mathcal M(\R^6)$ be such that
\[
\sup_n\iint_{\R^6}\left(\frac{|p|^2}{2}+\Phi_+(x)\right)\,dF_n\leq C
\]
and
\[
\lim_{n\to\infty}\iint_{\R^6}\varphi\,dF_n
=\iint_{\R^6}\varphi\,dF_0,\qquad \forall\varphi\in C_0(\R^6).
\]
Then $F_0\in\mathcal A$ and
\[
H(F_0)\leq \liminf_{n\to\infty}H(F_n),
\qquad
\|\nabla V_{F_0}\|_2^2\leq\liminf_{n\to\infty}\|\nabla V_{F_n}\|_2^2.
\]
\end{lemma}
\begin{proof}
The moment bound and the confinement of $\Phi$ imply tightness in both $x$ and $p$. Hence
$F_0(\R^6)=M$, and by lower semicontinuity of nonnegative lower semicontinuous functions,
\[
\iint_{\R^6}\left(\frac{|p|^2}{2}+\Phi_+(x)\right)\,dF_0\leq C.
\]

We next prove the potential lower semicontinuity. If
$\liminf_n\|\nabla V_{F_n}\|_2=\infty$, there is nothing to prove. Otherwise, after taking a
subsequence, $V_{F_n}\rightharpoonup V_0$ weakly in $\dot H^1$. For every
$\phi\in C_c^\infty(\R^3)$,
\[
\int_{\R^3}\nabla V_0\cdot\nabla\phi\,dx
=\lim_{n\to\infty}4\pi\int_{\R^3}\phi\,d\rho_{F_n}
=4\pi\int_{\R^3}\phi\,d\rho_{F_0}.
\]
Thus $-\Delta V_0=4\pi\rho_{F_0}$. Therefore $V_0=V_{F_0}$ in $\dot H^1$, and the weak lower semicontinuity of the norm gives the desired estimate.

It remains to discuss the entropy. We give the argument carefully because the convergence is
weak-* convergence in the space of finite Radon measures, not weak convergence in $L^1$.
Recall from \eqref{jdef} that $j=-h$ is convex on $[0,\infty)$, satisfies
$j(0)=0$ and $j\leq0$, and has recession function
\[
j_\infty=\lim_{s\to\infty}\frac{j(s)}s=0.
\]
We shall prove
\[
\iint_{\R^6}j(F_0^a)\,dpdx
\leq \liminf_{n\to\infty}\iint_{\R^6}j(F_n^a)\,dpdx.
\]
Together with $H(F)=\iint j(F^a)\,dpdx$, this is the desired lower semicontinuity of the entropy.

Put
\[
W(x,p):=1+\frac{|p|^2}{2}+\Phi_+(x).
\]
The assumptions give
\[
\sup_n\iint_{\R^6}W\,dF_n<\infty,
\qquad
\iint_{\R^6}W\,dF_0<\infty.
\]
For $R>0$, define
\[
\Omega_R:=\{(x,p)\in\R^6: W(x,p)<R\}.
\]
Since $\Phi$ is confining, $\Omega_R$ is bounded. We choose $R$ so that
$F_0(\partial\Omega_R)=0$; this excludes at most countably many values of $R$. Then the weak-*
convergence of $F_n$ to $F_0$ implies
\[
F_n\llcorner\Omega_R\stackrel{*}{\rightharpoonup}F_0\llcorner\Omega_R
\quad\text{in }\mathcal M(\Omega_R),
\]
where $F\llcorner\Omega_R$ denotes the restriction of the measure $F$
to $\Omega_R$, i.e., $(F\llcorner\Omega_R)(A):= F(A\cap\Omega_R)$ for every Borel set $A\subset\R^6$.

We first prove lower semicontinuity on this bounded set. For $m>0$, let
\[
j_m(s):=\max\{j(s),-m\}.
\]
Then $j_m$ is convex, continuous, bounded from below, and has recession
function $(j_m)_\infty=0$. Hence, by the lower-semicontinuity theory for convex functions of measures on bounded domains \cite[Theorem~1.1 and Proposition~1.2]{DemengelTemam}, we obtain
\[
\iint_{\Omega_R}j_m(F_0^a)\,dpdx
\leq
\liminf_{n\to\infty}\iint_{\Omega_R}j_m(F_n^a)\,dpdx.
\]
Equivalently, defining the convex conjugate
$j_m^*(q):=\sup_{s\geq0}\{qs-j_m(s)\}$ for $q\in\R$, this follows from the Fenchel representation
\[
\int_{\Omega_R}j_m(\nu^a)\,dpdx
=
\sup_{\psi\in C_c(\Omega_R)}
\left\{
\int_{\Omega_R}\psi\,d\nu-
\int_{\Omega_R}j_m^*(\psi)\,dpdx
\right\},
\]
valid for nonnegative measures $\nu$, because the singular part contributes the recession term
$(j_m)_\infty\nu^s(\Omega_R)=0$.

We now let $m\to\infty$. Since
\[
j_m(s)-j(s)=(h(s)-m)_+,
\]
and $h(s)/s\to0$ as $s\to\infty$, for every $\varepsilon>0$ there exists $m_\varepsilon>0$ such
that
\[
(h(s)-m)_+\leq\varepsilon s,
\qquad s\geq0,
\qquad m\geq m_\varepsilon.
\]
Therefore
\[
\sup_n\iint_{\Omega_R}(j_m(F_n^a)-j(F_n^a))\,dpdx
\leq \varepsilon M.
\]
Letting $m\to\infty$ in the preceding local lower semicontinuity inequality gives
\[
\iint_{\Omega_R}j(F_0^a)\,dpdx
\leq
\liminf_{n\to\infty}\iint_{\Omega_R}j(F_n^a)\,dpdx.
\]

It remains to remove the restriction to $\Omega_R$. By Lemma \ref{lem-decay}, with
$a=|p|+\sqrt{\Phi_+(x)}$,
\[
h(F_n^a)
\leq
\bigl(2+|p|+\sqrt{\Phi_+(x)}\bigr)F_n^a
+2e^{-\frac{|p|+\sqrt{\Phi_+(x)}}2}.
\]
On $\Omega_R^c=\{W\geq R\}$ we have
\[
2+|p|+\sqrt{\Phi_+(x)}\leq C\sqrt{W(x,p)}\leq \frac{C}{\sqrt R}W(x,p).
\]
Consequently,
\[
\sup_n\iint_{\Omega_R^c}
\bigl(2+|p|+\sqrt{\Phi_+(x)}\bigr)\,dF_n
\leq
\frac{C}{\sqrt R}
\sup_n\iint_{\R^6}W\,dF_n
\longrightarrow0
\]
as $R\to\infty$. Moreover,
\[
\iint_{\Omega_R^c}e^{-\frac{|p|+\sqrt{\Phi_+(x)}}2}\,dpdx
\longrightarrow0
\]
by dominated convergence and the confinement assumption
$\int_{\R^3}e^{-\sqrt{\Phi_+}/2}\,dx<\infty$. Thus
\[
\lim_{R\to\infty}\sup_n\iint_{\Omega_R^c}h(F_n^a)\,dpdx=0.
\]
The same estimate also holds with $F_0$ in place of $F_n$.

Given $\varepsilon>0$, choose an admissible $R$ so large that
\[
\sup_n\iint_{\Omega_R^c}h(F_n^a)\,dpdx\leq\varepsilon.
\]
Since $j=-h$, we have
\[
\iint_{\R^6}j(F_n^a)\,dpdx
\geq
\iint_{\Omega_R}j(F_n^a)\,dpdx-\varepsilon.
\]
Taking the lower limit and using the local lower semicontinuity gives
\[
\liminf_{n\to\infty}\iint_{\R^6}j(F_n^a)\,dpdx
\geq
\iint_{\Omega_R}j(F_0^a)\,dpdx-\varepsilon.
\]
Finally, since $h(F_0^a) \geq 0$,
\[
\iint_{\Omega_R}j(F_0^a)\,dpdx
=
\iint_{\R^6}j(F_0^a)\,dpdx
+
\iint_{\Omega_R^c}h(F_0^a)\,dpdx
\geq
\iint_{\R^6}j(F_0^a)\,dpdx.
\]
Therefore
\[
\liminf_{n\to\infty}\iint_{\R^6}j(F_n^a)\,dpdx
\geq
\iint_{\R^6}j(F_0^a)\,dpdx-\varepsilon.
\]
Letting $\varepsilon\downarrow0$ proves the entropy lower semicontinuity.
\end{proof}

\subsection{Existence of minimizers}
\begin{proof}[Proof of Proposition \ref{prop-existence-minimizer}]
Let $\{F_n\}\subset\mathcal A$ be a minimizing sequence. Since $\Phi_-$ is bounded and
\eqref{decay} holds, for every $\varepsilon>0$ there exists $C_\varepsilon>0$ such that
\[
2+|p|+\sqrt{\Phi_+(x)}
\leq \varepsilon\left(\frac{|p|^2}{2}+\Phi_+(x)\right)+C_\varepsilon.
\]
Hence
\[
H(F_n)\geq
-\varepsilon\iint_{\R^6}\left(\frac{|p|^2}{2}+\Phi_+(x)\right)\,dF_n-C_\varepsilon M-C_T,
\]
where
\[
C_T=2\iint_{\R^6}e^{-\frac{|p|+\sqrt{\Phi_+(x)}}2}\,dpdx<\infty.
\]
Choosing $\varepsilon>0$ sufficiently small and using that $\mathcal F(F_n)$ is bounded from above,
we obtain
\[
\sup_n\iint_{\R^6}\left(\frac{|p|^2}{2}+\Phi_+(x)\right)\,dF_n<\infty,
\qquad
\sup_n\|\nabla V_{F_n}\|_2<\infty.
\]
By tightness, after passing to a subsequence, $F_n$ converges weakly-* in $\mathcal M(\R^6)$ to some
$F_0$ with mass $M$. Lemma \ref{lsc} gives $F_0\in\mathcal A$ and
\[
\mathcal F(F_0)\leq\liminf_{n\to\infty}\mathcal F(F_n)=\mathcal F_{\min}.
\]
Thus $F_0$ is a minimizer.

Before proving uniqueness, observe that every minimizer has its singular part at zero momentum.
Indeed, write $F_0=f_0\,dxdp+\sigma_0$ and let $\rho_{\sigma_0}$ be the spatial projection of
$\sigma_0$. Replacing $\sigma_0$ by
$\rho_{\sigma_0}\otimes\delta_{p=0}$ preserves the mass, external energy, Coulomb energy, and
entropy, while it changes the kinetic energy by
\[
 -\iint_{\R^6}\frac{|p|^2}{2}\,d\sigma_0(x,p)\leq0.
\]
Minimality forces this integral to vanish. Hence
\[
 \sigma_0=\rho_{\sigma_0}\otimes\delta_{p=0}.
\]

We next prove uniqueness. Let $F_1$ and $F_2$ be two minimizers and set
$F_\theta=\theta F_1+(1-\theta)F_2$, $0<\theta<1$. The admissible class is convex. The linear
part of the energy is affine, the entropy part is convex because $j=-h$ is strictly convex on the
absolutely continuous density, and the Coulomb term is convex because
\[
\|\nabla V_{F_\theta}\|_2^2
=\|\theta\nabla V_{F_1}+(1-\theta)\nabla V_{F_2}\|_2^2.
\]
Since both endpoints minimize, equality must hold in all convexity inequalities. Equality in the
Coulomb term gives $\nabla V_{F_1}=\nabla V_{F_2}$ and hence $\rho_{F_1}=\rho_{F_2}$. Equality in
the strictly convex entropy term gives $F_1^a=F_2^a$ a.e. The preceding zero-momentum observation
applies to both singular parts. Since the total spatial densities and the absolutely continuous
parts agree, their singular spatial projections agree as well, and therefore $F_1=F_2$.

Finally assume that $\Phi$ is radial. If $R\in SO(3)$, the push-forward of $F_0$ by
$(x,p)\mapsto(Rx,Rp)$ is again admissible and has the same free energy, because the kinetic energy,
the external potential, the entropy, and the Coulomb energy are all invariant under this simultaneous
rotation. By uniqueness, this push-forward coincides with $F_0$. This proves radiality.
\end{proof}

\begin{prop}\label{prop-ac-nonzero}
For any minimizer $F_0$ of \eqref{mp-boson} with $T>0$, the absolutely continuous part $F_0^a$ does not vanish.
\end{prop}
\begin{proof}
Assume by contradiction that $F_0^a=0$. Choose a nonnegative smooth compactly supported density
$f\in C_c^\infty(\R^6)$ with $\iint f\,dpdx=M$ and finite potential energy, and set
\[
F_\varepsilon=(1-\varepsilon)F_0+\varepsilon f\,dpdx,
\qquad 0<\varepsilon<1.
\]
Then $F_\varepsilon\in\mathcal A$. The kinetic, external, and Coulomb parts change by $O(\varepsilon)$.
On the other hand,
\[
H(F_\varepsilon)=\iint \left(\varepsilon f\log(\varepsilon f)-(1+\varepsilon f)\log(1+\varepsilon f)\right)\,dpdx
=\varepsilon M\log\varepsilon+O(\varepsilon).
\]
Since $T>0$, the term $T\varepsilon M\log\varepsilon$ is negative and dominates the $O(\varepsilon)$
terms as $\varepsilon\downarrow0$. Hence $\mathcal F(F_\varepsilon)<\mathcal F(F_0)$ for
sufficiently small $\varepsilon$, contradicting the minimality of $F_0$.
\end{proof}

\begin{lemma}\label{lem-Fmin-strict-T}
The minimum free-energy level $\mathcal F_{\min}(T)$ is strictly decreasing in $T>0$.
\end{lemma}
\begin{proof}
Let $0<T_1<T_2$ and let $F_1$ be a minimizer at temperature $T_1$. By Proposition
\ref{prop-ac-nonzero}, $F_1^a\not\equiv0$, and hence
\[
H(F_1)=\iint\left(F_1^a\log F_1^a-(1+F_1^a)\log(1+F_1^a)\right)\,dpdx<0.
\]
Therefore
\[
\mathcal F_{\min}(T_2)\leq \mathcal F_{T_2}(F_1)
=\mathcal F_{T_1}(F_1)+(T_2-T_1)H(F_1)
<\mathcal F_{\min}(T_1).
\]
\end{proof}

\subsection{First variation in the measure class}

The characterization of the singular part of the free-energy minimizer
will lead to an obstacle problem. We record here the local regularity
result for the classical obstacle problem that will be used below.

\begin{prop}[{\cite{CaffarelliSalsa2005,Figalli2018Obstacle,PSU2012}}]
\label{P.obstacle}
Let $\Omega\subset\R^d$ be open, $\varphi\in W^{2,\infty}_{\rm loc}(\Omega)$,
and $g\in L^\infty_{\rm loc}(\Omega)$. Suppose that
$v\in H^1_{\rm loc}(\Omega)$ satisfies $v\geq\varphi$ q.e.\ in $\Omega$ and,
for every ball $B\Subset\Omega$,
\[
\int_B \nabla v\cdot\nabla(w-v)\,dx
-
\int_B g(w-v)\,dx
\geq0
\]
for every $w\in v+H_0^1(B)$ satisfying $w\geq\varphi$ q.e.\ in $B$.
Then the Lewy--Stampacchia inequality
\[
0\leq-\Delta v-g
\leq(-\Delta\varphi-g)_+
\qquad\text{a.e. in }\Omega
\]
holds. Consequently, $v\in W^{2,p}_{\rm loc}(\Omega)$ for every $1<p<\infty$.

If, in addition, $g\in C^{0,\alpha}_{\rm loc}(\Omega)$ for some
$\alpha>0$ and $\varphi\in C^{1,1}_{\rm loc}(\Omega)$, then $v\in C^{1,1}_{\rm loc}(\Omega)$.
\end{prop}

\begin{proof}
The Lewy--Stampacchia inequality and the resulting local
$W^{2,p}$ regularity are standard; see, for instance,
Caffarelli--Salsa~\cite{CaffarelliSalsa2005} and
Petrosyan--Shahgholian--Uraltseva~\cite{PSU2012}.
For the last assertion, locally solve $-\Delta z=g$.
Then $z\in C^{2,\alpha}_{\rm loc}$, and $v-z$ solves the zero-source
obstacle problem with obstacle $\varphi-z\in C^{1,1}_{\rm loc}$.
The classical optimal-regularity theorem
\cite[Theorem~4.1]{Figalli2018Obstacle} therefore yields
$v-z\in C^{1,1}_{\rm loc}$, and hence $v\in C^{1,1}_{\rm loc}$.
\end{proof}

\begin{proof}[Proof of Theorem~\ref{prop-EL-minimizer}]
Write
\[
 F_0=f\,dxdp+\sigma,
 \qquad f=F_0^a,
 \qquad \sigma=F_0^s,
\]
and set
\[
 e_0(x,p):=\frac{|p|^2}{2}+\Phi(x)+\widetilde V_{F_0}(x).
\]
All pointwise statements involving \(\widetilde V_{F_0}\) are understood through the
quasi-continuous representative. Since \(F_0\in\mathcal A\), its spatial density has finite Coulomb
energy. Hence it charges no set of zero Newtonian capacity. The same is true for the spatial
projection of every positive measure dominated by \(F_0\).

We first record the first-variation identities that will be used repeatedly. Let \(\mu\) be a signed
finite measure on phase space, with zero total mass and with spatial marginal \(\rho_\mu\) of finite
Coulomb energy, such that \(F_0+\varepsilon\mu\) is nonnegative for one-sided or two-sided small
\(\varepsilon\). Then
\[
V_{F_0+\varepsilon\mu}=V_{F_0}+\varepsilon V_\mu
\]
in \(\dot H^1\), and therefore
\[
\frac1{8\pi}\int_{\R^3} |\nabla V_{F_0+\varepsilon\mu}|^2  dx
=
\frac1{8\pi}\int_{\R^3} |\nabla V_{F_0}|^2 dx
+
\varepsilon\int_{\R^3} \widetilde V_{F_0}\,d\rho_\mu
+
\frac{\varepsilon^2}{8\pi}\int_{\R^3} |\nabla V_\mu|^2 dx.
\]
Consequently, the first variation of the total energy in the direction \(\mu\) is
\[
\int_{\R^6} e_0\,d\mu.
\]
For absolutely continuous variations, we shall also use $j$ defined in \eqref{jdef}, with
$j'(s)=\log\frac{s}{1+s}$ for $s>0$. Thus, \(H(F)=\iint j(F^a)\,dpdx\).

We next show that the regular density is strictly positive a.e. Suppose that \(f=0\) on a set of
positive Lebesgue measure. Since \(e_0\) is finite a.e., there exists a bounded measurable set
\(A\subset\{f=0\}\) of positive measure on which \(e_0\) is bounded. By Proposition
\ref{prop-ac-nonzero}, \(f\not\equiv0\). Hence we can choose a bounded compactly supported density
\(b\ge0\), supported in a bounded set on which
\(\delta\leq f\leq\delta^{-1}\) and \(|e_0|\leq R\) for some \(\delta,R>0\), such that
\[
0\le b\le f,
\qquad
m:=\iint b\,dpdx>0.
\]
Let \(a=m|A|^{-1}{\bf 1}_A\). For all sufficiently small \(\varepsilon>0\),
\[
 F_\varepsilon=F_0+\varepsilon a\,dxdp-\varepsilon b\,dxdp
\]
is admissible and has the same mass as \(F_0\). The energy variation is \(O(\varepsilon)\). The entropy
change on the set \(A\), where the original regular density is zero, is
\[
T\int_A j(\varepsilon a)\,dpdx
=Tm\varepsilon\log\varepsilon+O(\varepsilon),
\]
whereas the entropy change coming from removing \(\varepsilon b\) is only \(O(\varepsilon)\), because
\(b\) is supported where \(f\) is bounded away from zero. Hence
\(\mathcal F_T(F_\varepsilon)<\mathcal F_T(F_0)\) for small \(\varepsilon>0\), a contradiction. Thus
\[
 f>0\qquad\text{a.e. in }\R^6.
\]

We now derive the Euler--Lagrange equation on the absolutely continuous part. Define the nested
measurable sets
\[
 D_n:=\left\{(x,p): |x|+|p|\leq n,\quad n^{-1}\leq f(x,p)\leq n,\quad
 |e_0(x,p)|\leq n\right\}.
\]
They exhaust $\R^6$ up to a null set because $0<f<\infty$ and $|e_0|<\infty$ a.e. If
$\psi\in L^\infty(\R^6)$ is supported in $D_n$ and has zero integral, then, for sufficiently small
$|\varepsilon|$, the two-sided variation
\[
 F_\varepsilon=(f+\varepsilon\psi)\,dxdp+\sigma
\]
is admissible. Since $j'$ is bounded on $[n^{-1}/2,2n]$, differentiation under the integral is
justified and gives
\[
 \iint_{D_n}\left(e_0(x,p)+Tj'(f(x,p))\right)\psi(x,p)\,dpdx=0.
\]
It follows from the elementary zero-mean test-function characterization of constants that
$e_0+Tj'(f)$ is constant a.e. on every $D_n$ of positive measure. The constants agree on nested
sets, and the exhaustion therefore gives a constant \(c\in\R\) such that
\[
 e_0(x,p)+Tj'(f(x,p))+c=0
 \qquad\text{for a.e. }(x,p)\in\R^6.
\]
Solving this equation yields
\[
 f(x,p)=
\frac{1}{\exp\left(\frac1T\left(e_0(x,p)+c\right)\right)-1}
=
\frac{1}{\exp\left(\frac1T\left(\frac{|p|^2}{2}+\Phi(x)+\widetilde V_{F_0}(x)+c\right)\right)-1}
\]
for a.e. \((x,p)\).

We now prove the obstacle inequality. Define
\[
 u_0(x):=\Phi(x)+\widetilde V_{F_0}(x)+c.
\]
Assume, to the contrary, that \(\{u_0<0\}\) has positive Newtonian capacity. Since \(u_0\) is
quasi-continuous, there are \(\delta>0\) and a compact set
\(K\subset\{u_0\le -\delta\}\) with positive capacity. By the capacitary characterization, there is a
nonzero finite positive measure \(\nu_x\) supported on \(K\) with finite Coulomb energy. We scale
\(\nu_x\) so that its total mass is \(m>0\), where \(m\) is small enough that one can choose a bounded
compactly supported density \(b\), supported in some $D_n$, satisfying
\[
0\le b\le f,
\qquad
\iint_{\R^6}b\,dpdx=m.
\]
Set \(\nu=\nu_x\otimes\delta_{p=0}\). For \(0<\varepsilon\ll1\),
\[
 F_\varepsilon=F_0+\varepsilon\nu-\varepsilon b\,dxdp
\]
is admissible. Since the singular direction has no entropy contribution and the absolutely
continuous part changes from \(f\) to \(f-\varepsilon b\), Taylor's formula for $j$ on the bounded
range of $f$ over $\operatorname{supp}b$, together with the quadratic Coulomb expansion, gives
\[
\begin{aligned}
\mathcal F_T(F_\varepsilon)-\mathcal F_T(F_0)
&=\varepsilon\int_{\R^6}e_0\,d\nu
 -\varepsilon\iint_{\R^6} e_0 b\,dpdx
 -\varepsilon T\iint_{\R^6}j'(f)b\,dpdx
 +o(\varepsilon)  \\
&=\varepsilon\left(\int_{\R^6}e_0\,d\nu
 -\iint_{\R^6}(e_0+Tj'(f))b\,dpdx\right)+o(\varepsilon).
\end{aligned}
\]
Using \(e_0+Tj'(f)=-c\) and \(\iint b=m=\nu(\R^6)\), this becomes
\[
\mathcal F_T(F_\varepsilon)-\mathcal F_T(F_0)
=\varepsilon\left(\int_{\R^6}e_0\,d\nu+cm\right)+o(\varepsilon).
\]
Because \(\nu\) is supported on \(p=0\),
\[
 \int_{\R^6}e_0\,d\nu+cm
 =\int_K\bigl(\Phi+\widetilde V_{F_0}+c\bigr)\,d\nu_x
 =\int_K u_0\,d\nu_x
 \le -\delta m.
\]
Thus \(\mathcal F_T(F_\varepsilon)<\mathcal F_T(F_0)\) for small \(\varepsilon>0\), contradicting
minimality. Therefore
\[
 u_0\ge0\qquad\text{q.e. in }\R^3.
\]
This proves the obstacle inequality in (i).

We next locate the singular part. Let \(\eta\) be any finite positive measure with \(0\le\eta\le\sigma\),
and let \(m_\eta=\eta(\R^6)\). Choose a nonnegative bounded compactly supported density \(b\) with
\(\iint b\,dpdx=m_\eta\), supported in some $D_n$. Then
for \(0<\varepsilon\ll1\),
\[
 F_\varepsilon=F_0-\varepsilon\eta+\varepsilon b\,dxdp
\]
is admissible. Expanding as before, now with singular mass removed and absolutely continuous mass
added, gives
\[
\begin{aligned}
\mathcal F_T(F_\varepsilon)-\mathcal F_T(F_0)
&=-\varepsilon\int_{\R^6}e_0\,d\eta
 +\varepsilon\iint_{\R^6}e_0 b\,dpdx
 +\varepsilon T\iint_{\R^6}j'(f)b\,dpdx
 +o(\varepsilon) \\
&=\varepsilon\left(-\int_{\R^6}e_0\,d\eta
 +\iint_{\R^6}(e_0+Tj'(f))b\,dpdx\right)+o(\varepsilon) \\
&=\varepsilon\left(-\int_{\R^6}e_0\,d\eta-cm_\eta\right)+o(\varepsilon)
=-\varepsilon\int_{\R^6}(e_0+c)\,d\eta+o(\varepsilon).
\end{aligned}
\]
Minimality implies
\[
0\le -\int_{\R^6}(e_0+c)\,d\eta.
\]
On the other hand, the spatial projection of \(\eta\) charges no capacity-zero set, and \(u_0\ge0\)
q.e. Therefore
\[
 e_0(x,p)+c=\frac{|p|^2}{2}+u_0(x)\ge0
 \qquad \eta\text{-a.e.}
\]
The last two inequalities force
\[
 e_0(x,p)+c=0\qquad \eta\text{-a.e.}
\]
Since \(\eta\le\sigma\) was arbitrary, we obtain
\[
 e_0(x,p)+c=0\qquad \sigma\text{-a.e.}
\]
Both terms \(|p|^2/2\) and \(u_0(x)\) are nonnegative \(\sigma\)-a.e.; hence both vanish
\(\sigma\)-a.e. Thus \(p=0\) \(\sigma\)-a.e. and \(u_0=0\) for the spatial projection of \(\sigma\).
Consequently
\[
 F_0^s=\rho_0^s\otimes\delta_{p=0},
 \qquad
 u_0=0\quad \rho_0^s\text{-a.e.}
\]
This proves (ii), except for compactness of the support. Since \(u_0=0\) \(\rho_0^s\)-a.e. and
\(V_{F_0}\ge0\), we have \(\Phi\le -c\) \(\rho_0^s\)-a.e. The set \(\{x:\Phi(x)>-c\}\) is open and has zero
\(\rho_0^s\)-measure; hence no point of \(\operatorname{supp}\rho_0^s\) lies in it. Therefore
\[
 \operatorname{supp}\rho_0^s\subset\{x:\Phi(x)\le -c\},
\]
and the right-hand side is compact by the confinement assumption on \(\Phi\).

It remains to prove (iii). Assume $\Phi\in C^2_{\rm loc}(\R^3)$ and $\rho_0^s\neq0$. Define
\[
 \psi:=-\Phi-c
 \qquad\text{and}\qquad
 R_T(s):=\int_{\R^3}
 \frac{1}{\exp\left(\frac1T\left(\frac{|p|^2}{2}+s\right)\right)-1}\,dp,
 \quad s\geq0.
\]
Thus $u_0=V_{F_0}-\psi\geq0$ q.e. and
\[
 \rho_0^a=R_T(u_0)\qquad\text{a.e. in }\R^3.
\]
We first derive the precise obstacle variational inequality. Let $B\Subset\R^3$, and let
$W\in V_{F_0}+H_0^1(B)$ satisfy $W\geq\psi$ q.e. in $B$, with $W$ extended by $V_{F_0}$
outside $B$. Since $0\leq\rho_0^s\leq\rho_{F_0}$, the measure $\rho_0^s$ has finite Coulomb
energy and the pairing with the quasi-continuous representative of $W-V_{F_0}$ is well defined.
The distributional Poisson equation and Proposition~\ref{prop-quasicont} give
\begin{align}
&\int_B\nabla V_{F_0}\cdot\nabla(W-V_{F_0})\,dx
-4\pi\int_B R_T(u_0)(W-V_{F_0})\,dx \notag\\
&\hspace{4cm}
=4\pi\int_B(\widetilde W-\widetilde V_{F_0})\,d\rho_0^s\geq0.
\label{local-obstacle-V}
\end{align}
Indeed, $V_{F_0}=\psi$ $\rho_0^s$-a.e. by (ii), whereas
$\widetilde W\geq\psi$ outside a set of capacity zero, and $\rho_0^s$ does not charge such sets.
Hence $V_{F_0}$ is, on every relatively compact set, the solution of the classical obstacle
problem with obstacle $\psi$ and fixed source
\[
 G_0(x):=4\pi R_T(u_0(x)).
\]

We now apply Proposition~\ref{P.obstacle} to \eqref{local-obstacle-V} with $v=V_{F_0}$, $\varphi=-\Phi-c$, $g=G_0$.
Indeed, $\varphi\in C^2_{\rm loc}(\R^3)\subset
W^{2,\infty}_{\rm loc}(\R^3)$ and
$0\leq G_0\leq4\pi R_T(0)$. Hence the Lewy--Stampacchia inequality gives
\begin{equation}\label{LS-obstacle-V}
0\leq-\Delta V_{F_0}-G_0
\leq\bigl(\Delta\Phi-G_0\bigr)_+
\qquad\text{a.e. locally},
\end{equation}
and
\[
V_{F_0}\in W^{2,p}_{\rm loc}(\R^3)
\qquad\text{for every }1<p<\infty.
\]
Since $-\Delta V_{F_0}-G_0=4\pi\rho_0^s$ as distributions,
\eqref{LS-obstacle-V} also shows that $\rho_0^s$ has a locally bounded
density.

%

We now bootstrap to optimal regularity. Expanding the Bose factor into a geometric series gives
\[
 R_T(s)=(2\pi T)^{3/2}\sum_{k=1}^{\infty}k^{-3/2}e^{-ks/T}.
\]
Using $1-e^{-a}\leq\min\{a,1\}$ and splitting the series at
$k\sim T/|s-t|$, we obtain
\begin{equation}\label{RT-holder}
 |R_T(s)-R_T(t)|\leq C_T|s-t|^{1/2},
 \qquad s,t\geq0.
\end{equation}
Taking $p>3$ in the preceding $W^{2,p}_{\rm loc}$ estimate shows that
$V_{F_0}$, and hence $u_0=\Phi+V_{F_0}+c$, is locally Lipschitz. It follows from
\eqref{RT-holder} that $G_0\in C^{0,1/2}_{\rm loc}$. The local optimal-regularity theorem applied
once more to \eqref{local-obstacle-V} therefore gives
\[
 V_{F_0}\in C^{1,1}_{\rm loc}(\R^3),
 \qquad
 u_0\in C^{1,1}_{\rm loc}(\R^3).
\]
We henceforth use these continuous representatives and set
\[
 \mathcal C:=\{x:u_0(x)=0\},
 \qquad
 \mathcal N:=\{x:u_0(x)>0\}.
\]
The contact condition in (ii) now implies
$\operatorname{supp}\rho_0^s\subset\mathcal C$.

For a $C^{1,1}$ function, its Hessian vanishes a.e. on any level set on which the function is
constant. Consequently,
\[
 D^2u_0=0\qquad\text{a.e. on }\mathcal C.
\]
On $\mathcal C$, we also have $R_T(u_0)=R_T(0)$, whereas $\rho_0^s$ vanishes on the open set
$\mathcal N$. Using
\[
 \Delta u_0=\Delta\Phi-4\pi R_T(u_0)-4\pi\rho_0^s
\]
and the absolute continuity already obtained from \eqref{LS-obstacle-V}, we conclude that, as
measures on all of $\R^3$,
\begin{equation}\label{rho-s-density-contact}
 d\rho_0^s
 =\left(\frac{1}{4\pi}\Delta\Phi-R_T(0)\right){\bf 1}_{\mathcal C}\,dx
 =\left(\frac{1}{4\pi}\Delta\Phi-\rho_0^a\right){\bf 1}_{\mathcal C}\,dx.
\end{equation}
Since $\rho_0^a=R_T(u_0)$ and $R_T$ is continuous on $[0,\infty)$, the function
\[
 g_0:=\frac{1}{4\pi}\Delta\Phi-\rho_0^a
\]
is continuous on $\R^3$. Thus \eqref{rho-s-density-contact} is precisely the global identity 
asserted in (iii). 

It remains to show that $\operatorname{int}\mathcal C$ is nonempty. We prove the stronger fact that
$\operatorname{int}\operatorname{supp}\rho_0^s$ is nonempty. Set
\[
 q(x):=\Delta\Phi(x)-4\pi R_T(0),
 \qquad
 \Omega_+:=\{x:q(x)>0\}.
\]
Since $\rho_0^s\neq0$, \eqref{rho-s-density-contact} implies that
\[
 E:=\mathcal C\cap\Omega_+
\]
has positive Lebesgue measure. We claim that
\begin{equation}\label{local-free-boundary-null}
 |\partial\mathcal N\cap\Omega_+|=0.
\end{equation}
Only this nondegenerate part of the free boundary is needed; no assertion is made here about the
possibly degenerate set $\partial\mathcal N\cap\{q=0\}$.

To prove \eqref{local-free-boundary-null}, fix $K\Subset\Omega_+$. On a neighborhood of $K$ we have
$q\geq\delta>0$. In $\mathcal N$ the singular density vanishes, and hence
\[
 \Delta u_0=\Delta\Phi-4\pi R_T(u_0)
 \geq\Delta\Phi-4\pi R_T(0)=q\geq\delta.
\]
The standard maximum-principle argument gives the quadratic nondegeneracy estimate
\begin{equation}\label{quadratic-nondegeneracy}
 \sup_{B_r(x_0)}u_0\geq\frac{\delta}{6}r^2
\end{equation}
for $x_0\in\partial\mathcal N\cap K$ and all sufficiently small $r$. Indeed, for a point
$x\in\mathcal N$ one applies the maximum principle to
$u_0(y)-\delta|y-x|^2/6$ on
$B_r(x)\cap\mathcal N$, and then lets $x\to x_0$.

Let $L$ be a local Lipschitz constant for $\nabla u_0$. Since
$u_0(x_0)=0$, $u_0\geq0$, and $u_0\in C^1$, we have $\nabla u_0(x_0)=0$.
Applying \eqref{quadratic-nondegeneracy} in $B_{r/2}(x_0)$, we find
$y\in B_{r/2}(x_0)$ such that
\[
 u_0(y)\geq\frac{\delta}{24}r^2.
\]
With
\[
 a:=\min\left\{\frac14,\frac{\delta}{48(1+L)}\right\},
\]
the Lipschitz bound for $\nabla u_0$ gives
\[
 B_{ar}(y)\subset\mathcal N\cap B_r(x_0).
\]
Thus the positivity set occupies a fixed positive proportion of every sufficiently small ball
centered at a point of $\partial\mathcal N\cap K$. Such a point cannot be a Lebesgue density point
of $\mathcal C$. Since $\partial\mathcal N\subset\mathcal C$, the Lebesgue density theorem gives
$|\partial\mathcal N\cap K|=0$. Exhausting $\Omega_+$ by compact sets proves
\eqref{local-free-boundary-null}.

Because
$\mathcal C\setminus\operatorname{int}\mathcal C=\partial\mathcal N$, the set
$E\cap\operatorname{int}\mathcal C$ has positive measure. Choose
$x_0\in E\cap\operatorname{int}\mathcal C$. By continuity of $\Delta\Phi$, there exists $r>0$
such that
\[
 B_r(x_0)\subset\operatorname{int}\mathcal C\cap\Omega_+.
\]
The density in \eqref{rho-s-density-contact} is strictly positive on this ball, so
$B_r(x_0)\subset\operatorname{supp}\rho_0^s$. Therefore
\[
 \operatorname{int}\operatorname{supp}\rho_0^s\neq\varnothing.
\]
Since $\operatorname{supp}\rho_0^s\subset\mathcal C$, this also proves
$\operatorname{int}\mathcal C\neq\varnothing$.

Finally, restricting \eqref{rho-s-density-contact} to $\operatorname{int}\mathcal C$ and adding the
absolutely continuous and condensed spatial densities gives
\[
 \rho_{F_0}=\frac{1}{4\pi}\Delta\Phi
\]
in $\mathcal D'(\operatorname{int}\mathcal C)$. This completes the proof.
\end{proof}

\section{Sharp critical temperature at fixed mass}\label{sec-temp-transition}

We first prove that a normal phase exists at sufficiently high temperature and that condensation
occurs at sufficiently low temperature. We then formulate the problem at fixed chemical potential
as a convex obstacle problem. Its comparison principle shows that normality persists when the
temperature is increased. A closedness argument includes the transition point and completes the
proof of Theorem~\ref{thm-T}.

\subsection{Auxiliary normal states}
\begin{lemma}\label{lem-exist-ell}
For any $T>0$ and $c\geq-\Phi_{\min}$, there exists a unique solution
$V=V_{T,c}\in\dot H^1(\R^3)$ of
\begin{equation}\label{ell}
\left\{\begin{aligned}
-\Delta V &=4\pi\int_{\R^3}\frac{1}{e^{\frac{1}{T}(\frac{|p|^2}{2}+\Phi+V+c)}-1}\,dp,\\
V&>0,
\end{aligned}\right.
\qquad\text{in }\R^3.
\end{equation}
\end{lemma}
\begin{proof}
By replacing $\Phi$ and $c$ by $\Phi-\Phi_{\min}$ and $c+\Phi_{\min}$, respectively, we may assume
$\Phi_{\min}=0$ and $c\geq0$. Define
\[
g(x,t)=4\pi\int_{\R^3}\frac{1}{e^{\frac1T(\frac{|p|^2}{2}+\Phi(x)+t+c)}-1}\,dp,
\qquad
G(x,t)=\int_0^t g(x,s)\,ds.
\]
Consider
\[
I_{\min}=\inf\left\{ I(V)=\int_{\R^3}\frac12|\nabla V|^2-G(x,V)\,dx
~\middle|~ V\in\dot H^1(\R^3),\ V\geq0\right\}.
\]

Let $\{V_n\}$ be a minimizing sequence. Set
\[
\mathcal R=\left\{x\in\R^3~\middle|~\frac12e^{\Phi(x)/T}\leq1\right\}.
\]
This set is compact. On $\mathcal R^c$,
\[
g(x,t)\leq Ce^{-(\Phi(x)+t)/T},
\]
whereas on $\mathcal R$ we have $g(x,t)\leq C$. Consequently
\[
G(x,t)\leq C\left(t\chi_{\mathcal R}+e^{-\Phi/T}\chi_{\mathcal R^c}\right).
\]
By Sobolev embedding,
\[
\int_{\R^3}G(x,V_n)\,dx
\leq C(1+\|V_n\|_{L^6(\mathcal R)})
\leq C(1+\|\nabla V_n\|_2).
\]
Thus $\{V_n\}$ is bounded in $\dot H^1$.

After passing to a subsequence, $V_n\rightharpoonup V_0$ weakly in $\dot H^1$, strongly in
$L^1_{\rm loc}$, and a.e. The preceding bound and the mean value theorem imply, for every ball
$B_R\supset\mathcal R$,
\[
\int_{\R^3}|G(x,V_n)-G(x,V_0)|\,dx
\leq C\|V_n-V_0\|_{L^1(B_R)}
+C\|e^{-\Phi/T}\|_{L^{6/5}(B_R^c)}\|\nabla(V_n-V_0)\|_2.
\]
Letting first $n\to\infty$ and then $R\to\infty$ gives
\[
\int_{\R^3}G(x,V_n)\,dx\to\int_{\R^3}G(x,V_0)\,dx.
\]
Therefore $V_0$ minimizes $I$.

The Euler inequality in the nonnegative cone gives
\[
 \int_{\R^3}\nabla V_0\cdot\nabla\varphi\,dx-
 \int_{\R^3}g(x,V_0)\varphi\,dx\geq0
 \qquad\text{for every }\varphi\in C_c^\infty(\R^3),\ \varphi\geq0.
\]
Equivalently, $-\Delta V_0\geq g(x,V_0)$ in the sense of distributions. Since
$g(x,V_0)\geq0$ and $g(x,V_0)\not\equiv0$, the order-preserving property of the Newtonian inverse
implies
\[
 V_0\geq (-\Delta)^{-1}g(\cdot,V_0)>0.
\]
Indeed, the Newtonian kernel is strictly positive; in particular, on the support of any fixed test
function the right-hand side has a positive lower bound. Hence, for every compactly supported smooth
$\varphi$ and all sufficiently small $|\varepsilon|$, $V_0+\varepsilon\varphi$ remains nonnegative on
the support of $\varphi$. We may vary in both directions and obtain
\[
 \int_{\R^3}\nabla V_0\cdot\nabla\varphi\,dx=
 \int_{\R^3}g(x,V_0)\varphi\,dx,
 \qquad \varphi\in C_c^\infty(\R^3).
\]
Hence $V_0$ solves \eqref{ell}.

Finally, for every fixed $x$, the function $t\mapsto g(x,t)$ is strictly decreasing on its domain.
Thus $G(x,\cdot)$ is strictly concave and $-G(x,\cdot)$ is strictly convex. Consequently $I$ is
strictly convex on the nonnegative cone of $\dot H^1$, and the solution is unique.
\end{proof}

\begin{defi}\label{D.d.def}
For $F,F_0\in\mathcal A$, define
\[
d(F^a,F_0^a)=\iint_{\R^6}\left[-h(F^a)+h(F_0^a)+h'(F_0^a)(F^a-F_0^a)\right]dpdx.
\]
Then $d(F^a,F_0^a)\geq0$, and equality holds if and only if $F^a=F_0^a$ a.e., because $-h$ is
strictly convex.
\end{defi}

\begin{lemma}\label{lem-expansion}
Let $V_{T,c}\in\dot H^1(\R^3)$ be the solution of \eqref{ell} with $T>0$ and $c\geq-\Phi_{\min}$.
Denote
\[
F_{T,c}=\frac{1}{e^{\frac{1}{T}(\frac{|p|^2}{2}+\Phi+V_{T,c}+c)}-1},
\qquad
M_{T,c}=\iint_{\R^6}F_{T,c}\,dpdx.
\]
Let $F\in\mathcal A$ with $F(\R^6)=M_{T,c}$. Then
\begin{equation}\label{expansion}
\begin{aligned}
\mathcal F(F)-\mathcal F(F_{T,c})
&=\iint_{\R^6}\left(\frac{|p|^2}{2}+\Phi+V_{F_{T,c}}+c\right)dF^s
+T d(F^a,F_{T,c}) \\
&\quad +\frac{1}{8\pi}\|\nabla(V_F-V_{F_{T,c}})\|_2^2.
\end{aligned}
\end{equation}
In particular, $F_{T,c}$ is the unique minimizer of \eqref{mp-boson} with mass $M=M_{T,c}$.
\end{lemma}
\begin{proof}
We first verify that the quantities in the statement are finite. The proof of
Lemma~\ref{lem-exist-ell} gives $V_{T,c}>0$ locally and $V_{T,c}\geq0$ globally. Since
$c\geq-\Phi_{\min}$,
\[
 \Phi+V_{T,c}+c\geq\Phi-\Phi_{\min}\geq0.
\]
The exponential decay of the Bose factor for large one-particle energy, together with the standing
confinement assumption, implies
\[
 \iint_{\R^6}\left(1+\frac{|p|^2}{2}+\Phi_+(x)\right)F_{T,c}(x,p)\,dpdx<\infty
\]
and $H(F_{T,c})>-\infty$. The Poisson equation and
$V_{T,c}\in\dot H^1$ show that its spatial density has finite Coulomb energy. Hence
$F_{T,c}\,dxdp\in\mathcal A_{M_{T,c}}$.

Since $h'(s)=\log(1+1/s)$, the definition of $F_{T,c}$ gives
\[
\frac{|p|^2}{2}+\Phi+V_{F_{T,c}}+c=T h'(F_{T,c}).
\]
Using this identity, the equality of the masses of $F$ and $F_{T,c}$, and
\[
\frac{1}{8\pi}\left(\|\nabla V_F\|_2^2-\|\nabla V_{F_{T,c}}\|_2^2\right)
=\int_{\R^3}V_{F_{T,c}}\,d(\rho_F-\rho_{F_{T,c}})
+\frac{1}{8\pi}\|\nabla(V_F-V_{F_{T,c}})\|_2^2,
\]
we obtain exactly \eqref{expansion}. The first term on the right-hand side is nonnegative because
$c\geq-\Phi_{\min}$ and $V_{F_{T,c}}>0$. The other two terms are nonnegative by definition. Thus
$F_{T,c}$ is a minimizer.

If another minimizer $F_0$ has the same mass, all three nonnegative terms in \eqref{expansion} must
vanish. In particular $d(F_0^a,F_{T,c})=0$, hence $F_0^a=F_{T,c}$ a.e. The mass constraint then gives
$F_0^s=0$, and so $F_0=F_{T,c}$.
\end{proof}

\begin{lemma}\label{lem-continuity-MTc}
For each fixed $T>0$, the map $c\mapsto M_{T,c}$ is continuous on $[-\Phi_{\min},\infty)$, and
\[
\lim_{c\to\infty}M_{T,c}=0.
\]
\end{lemma}
\begin{proof}
For the proof set
\[
 R_T(s):=\int_{\R^3}\frac{1}{\exp\left(\frac1T(\frac{|p|^2}{2}+s)\right)-1}\,dp,
 \qquad s\geq0.
\]
Then
\[
 M_{T,c}=\int_{\R^3}R_T(\Phi+V_{T,c}+c)\,dx.
\]
Let $c_n\to c$ with $c_n\geq-\Phi_{\min}$. Since the sequence $\{c_n\}$ is contained in a bounded
$c$-interval, the estimates in the proof of Lemma \ref{lem-exist-ell} give a uniform $\dot H^1$ bound
for $V_n:=V_{T,c_n}$. Passing to a subsequence, we may assume
\[
 V_n\rightharpoonup V\quad\hbox{in }\dot H^1(\R^3),
 \qquad
 V_n\to V\quad\hbox{a.e. and in }L^1_{\rm loc}(\R^3).
\]
For every $\varphi\in C_c^\infty(\R^3)$,
\[
 \int_{\R^3}\nabla V_n\cdot\nabla\varphi\,dx
 =4\pi\int_{\R^3}R_T(\Phi+V_n+c_n)\varphi\,dx.
\]
On the support of $\varphi$ the integrand is bounded by $R_T(0)|\varphi|$, and the pointwise
convergence gives
\[
 R_T(\Phi+V_n+c_n)\to R_T(\Phi+V+c)
 \quad\hbox{a.e.}
\]
Thus we may pass to the limit in the weak formulation. The limit $V$ solves \eqref{ell} with parameter
$c$, and uniqueness in Lemma \ref{lem-exist-ell} implies $V=V_{T,c}$. Since every subsequence has the
same limit, the whole sequence converges in the above sense to $V_{T,c}$.

It remains to pass to the limit in the mass. The local part follows from dominated convergence: for
every ball $B_R$,
\[
 \int_{B_R} R_T(\Phi+V_n+c_n)\,dx
 \to
 \int_{B_R} R_T(\Phi+V_{T,c}+c)\,dx.
\]
The tails are uniform. Indeed, if $n$ is large then $c_n\geq c_-$ for some
$c_{-}>-\infty$ with $c_{-}\geq-\Phi_{\min}$, and $V_n\geq0$. Since $R_T$ is decreasing,
\[
 R_T(\Phi+V_n+c_n)\leq R_T(\Phi+c_{-}).
\]
For large values of $\Phi+c_-$ we have the elementary bound
$R_T(\Phi+c_{-})\leq C e^{-(\Phi+c_{-})/(2T)}$, and this function is integrable at infinity by the
standing confinement assumption on $\Phi$. Hence
\[
 \sup_n\int_{B_R^c}R_T(\Phi+V_n+c_n)\,dx\to0
 \qquad\text{as }R\to\infty.
\]
Combining the local convergence and the uniform tail estimate yields
$M_{T,c_n}\to M_{T,c}$.

Finally, since $V_{T,c}\geq0$ and $R_T$ is decreasing,
\[
 0\leq M_{T,c}\leq\int_{\R^3}R_T(\Phi+c)\,dx.
\]
The right-hand side tends to zero as $c\to\infty$ by the same local dominated-convergence and tail
argument. Therefore $M_{T,c}\to0$.
\end{proof}

\subsection{A preliminary high-temperature normal regime}
For fixed $T>0$ define
\[
 M^*(T):=\sup_{c\geq-\Phi_{\min}}M_{T,c}.
\]
We claim that
\[
\lim_{T\to\infty}M^*(T)=\infty.
\]
It is enough to show $M_{T,-\Phi_{\min}}\to\infty$. Suppose not. Then for a sequence $T_n\to\infty$,
with $V_n=V_{T_n,-\Phi_{\min}}$ and $\rho_n=-(4\pi)^{-1}\Delta V_n$, we would have
$M_n=\int\rho_n\,dx\leq M_0$. Moreover
\[
\rho_n(x)\leq C T_n^{3/2},
\]
and the standard Newton potential estimate gives
\[
\|V_n\|_\infty\leq C\|\rho_n\|_1^{2/3}\|\rho_n\|_\infty^{1/3}
\leq C M_0^{2/3}T_n^{1/2}.
\]
Choose a ball $B$ of positive measure on which $\Phi-\Phi_{\min}\leq1$. For $x\in B$ and
$|p|^2/2\leq T_n$, the quantity
\[
\frac1{T_n}\left(\frac{|p|^2}{2}+\Phi(x)+V_n(x)-\Phi_{\min}\right)
\]
is uniformly bounded. Therefore the Bose factor is bounded from below by a positive constant on
$B\times\{|p|^2/2\leq T_n\}$. Hence
\[
M_n\geq c |B| T_n^{3/2}\to\infty,
\]
which contradicts $M_n\leq M_0$. Thus $M^*(T)\to\infty$.

Let the prescribed mass $M>0$ be fixed. By the preceding divergence, $M<M^*(T)$ for
all sufficiently large $T$. Fix such a $T$. By the definition of $M^*(T)$, there is
$c_1\geq-\Phi_{\min}$ such that
$M<M_{T,c_1}$. On the other hand, Lemma \ref{lem-continuity-MTc} gives
$M_{T,c}\to0$ as $c\to\infty$. Since $c\mapsto M_{T,c}$ is continuous on
$[-\Phi_{\min},\infty)$, the intermediate value theorem gives a number $c_0\geq c_1$ such that
\[
 M=M_{T,c_0}.
\]
By Lemma \ref{lem-expansion}, $F_{T,c_0}$ is the unique minimizer of \eqref{mp-boson} with mass $M$.
It is purely absolutely continuous, hence the singular spatial part of the minimizer vanishes:
\[
 \rho_T^s=0.
\]
Thus the fixed-mass normal-temperature set is nonempty.

\subsection{Normal-state energy level}
For the low-temperature comparison we isolate the class of purely absolutely continuous Bose normal
states with prescribed mass. Set
\[
 E_{\min}:=\inf_{F\in\mathcal A}\mathcal E(F),
 \qquad
 E_T:=\mathcal E(F_T),
\]
where $F_T$ denotes the minimizer of \eqref{mp-boson} at temperature $T>0$. For $T>0$, $c\in\R$, and
$V\in\dot H^1(\R^3)$, define
\[
 U_{c,V}(x):=\Phi(x)+V(x)+c
\]
and, whenever $U_{c,V}\geq0$ q.e.,
\[
 \alpha_{T,c,V}(x,p):=
 \frac{1}{\exp\left(\frac1T\left(\frac{|p|^2}{2}+U_{c,V}(x)\right)\right)-1}.
\]
Let $\mathcal B_M$ be the set of triples $(T,c,V)$ such that
\[
 T>0,\qquad V\in\dot H^1(\R^3),\qquad U_{c,V}\geq0\quad\text{q.e.},
\]
\[
 -\Delta V=4\pi\rho_{\alpha_{T,c,V}}
 \quad\text{in }\mathcal D'(\R^3),
 \qquad
 \rho_{\alpha_{T,c,V}}(x):=\int_{\R^3}\alpha_{T,c,V}(x,p)\,dp,
\]
and
\[
 \iint_{\R^6}\alpha_{T,c,V}(x,p)\,dpdx=M.
\]
For $(T,c,V)\in\mathcal B_M$, define
\[
 J(T,c,V):=\mathcal E(\alpha_{T,c,V}\,dxdp),
\]
that is,
\[
J(T,c,V)=
\iint_{\R^6}\left(\frac{|p|^2}{2}+\Phi(x)\right)
\alpha_{T,c,V}(x,p)\,dpdx
+\frac1{8\pi}\int_{\R^3}|\nabla V|^2\,dx.
\]
Finally set
\[
 J_{\min}:=\inf_{(T,c,V)\in\mathcal B_M}J(T,c,V),
\]
with the convention $J_{\min}=+\infty$ if $\mathcal B_M=\emptyset$.

\subsection{The zero-temperature energy problem}
\begin{lemma}\label{lem-Jmin-gap}
With the convention $\inf\emptyset=+\infty$, one has
\[
E_{\min}<J_{\min}.
\]
\end{lemma}
\begin{proof}
If $\mathcal B_M=\emptyset$, then $J_{\min}=+\infty$ and there is nothing to prove. Hence assume that
$J_{\min}<\infty$. Since $\mathcal E(F)\geq M\Phi_{\min}$ for every admissible $F$, we also have
$J_{\min}>-\infty$.

We first prove a uniform positive lower bound on the temperature for normal states whose total energy
is bounded. Fix $C>0$. We claim that there exists $\tau_C>0$ such that every
$(T,c,V)\in\mathcal B_M$ satisfying $J(T,c,V)\leq C$ has $T\geq\tau_C$.
Indeed, write
\[
U(x):=\Phi(x)+V(x)+c\geq0,
\qquad
\rho(x):=\int_{\R^3}\alpha_{T,c,V}(x,p)\,dp=R_T(U(x)).
\]
Since $R_T$ is decreasing,
\[
0\leq \rho(x)\leq R_T(0)=(2\pi T)^{3/2}\zeta(3/2).
\]
Moreover, since $\Phi_-$ is bounded and the kinetic and Coulomb energies are nonnegative,
\[
\int_{\R^3}\Phi_+(x)\rho(x)\,dx
\leq C+M\|\Phi_-\|_{L^\infty}=:C_1.
\]
Choose $R>0$ so large that $\Phi_+(x)\geq 4C_1/M$ for $|x|\geq R$. Then
\[
\int_{|x|\geq R}\rho(x)\,dx\leq \frac{M}{4}.
\]
On the other hand,
\[
\int_{|x|<R}\rho(x)\,dx
\leq |B_R|(2\pi T)^{3/2}\zeta(3/2).
\]
If $T>0$ is sufficiently small, the last quantity is at most $M/4$, contradicting the mass constraint
$\int\rho=M$. This proves the claim.

Next we record a lower bound on the kinetic energy of a Bose normal state. For $z\geq0$, set
\[
B(z):=\int_{\R^3}\frac{1}{e^{|q|^2/2+z}-1}\,dq,
\qquad
A(z):=\int_{\R^3}\frac{|q|^2/2}{e^{|q|^2/2+z}-1}\,dq.
\]
The ratio $A(z)/B(z)$ is continuous on $[0,\infty)$ and has a positive lower bound; indeed it has a
finite positive value at $z=0$ and tends to the Maxwellian value $3/2$ as $z\to\infty$. Hence there is
$\kappa>0$ such that $A(z)\geq\kappa B(z)$ for all $z\geq0$. By the scaling $p=\sqrt T q$, for each
$x$,
\[
\int_{\R^3}\frac{|p|^2}{2}\alpha_{T,c,V}(x,p)\,dp
=T^{5/2}A\left(\frac{U(x)}{T}\right)
\geq \kappa T^{5/2}B\left(\frac{U(x)}{T}\right)
=\kappa T\rho(x).
\]
Consequently, if $J(T,c,V)\leq C$, then
\[
\iint_{\R^6}\frac{|p|^2}{2}\alpha_{T,c,V}(x,p)\,dpdx
\geq \kappa T M\geq \kappa\tau_C M.
\]

Now take a minimizing sequence $(T_n,c_n,V_n)\in\mathcal B_M$ with
$J(T_n,c_n,V_n)\downarrow J_{\min}$. For all sufficiently large $n$ we have
$J(T_n,c_n,V_n)\leq C:=\max\{J_{\min}+1,1\}$. Let
\[
F_n(x,p):=\alpha_{T_n,c_n,V_n}(x,p),
\qquad
\rho_n(x):=\int_{\R^3}F_n(x,p)\,dp.
\]
Define the zero-momentum competitor
\[
\widetilde F_n:=\rho_n(x)\,dx\otimes\delta_0(p).
\]
Then $\widetilde F_n\in\mathcal A$, has mass $M$, and has the same spatial density as $F_n$. Hence
its external potential energy and Coulomb energy coincide with those of $F_n$, while its kinetic
energy is zero. Therefore
\[
E_{\min}
\leq \mathcal E(\widetilde F_n)
=J(T_n,c_n,V_n)-
\iint_{\R^6}\frac{|p|^2}{2}F_n(x,p)\,dpdx
\leq J(T_n,c_n,V_n)-\kappa\tau_C M.
\]
Letting $n\to\infty$ gives
\[
J_{\min}\geq E_{\min}+\kappa\tau_C M>E_{\min}.
\]
\end{proof}

\subsection{Convergence of energy levels as \texorpdfstring{$T\downarrow0$}{T tends to zero}}
\begin{lemma}\label{lem-ET-limit}
The map $T\mapsto E_T$ is strictly increasing, and
\[
\lim_{T\downarrow0}E_T=E_{\min}.
\]
\end{lemma}
\begin{proof}
Let $0<T_1<T_2$, and write $F_i=F_{T_i}$, $E_i=\mathcal E(F_i)$, and $H_i=H(F_i)$. By the minimizing
property,
\[
E_1+T_1H_1\leq E_2+T_1H_2,
\]
and
\[
E_2+T_2H_2\leq E_1+T_2H_1.
\]
Adding the two inequalities gives
\[
(T_2-T_1)(H_2-H_1)\leq0,
\]
so $H_2\leq H_1$. Substituting this into the first inequality yields $E_1\leq E_2$.

If equality $E_1=E_2$ held, then also $H_1=H_2$, and both inequalities above would be equalities.
Thus $F_1$ minimizes both $\mathcal F_{T_1}$ and $\mathcal F_{T_2}$. By uniqueness of minimizers,
$F_1=F_2$. Applying Theorem \ref{prop-EL-minimizer} at the two temperatures gives constants
$c_1,c_2$ such that, on the set where $F_1^a>0$,
\[
\frac{|p|^2}{2}+\Phi+V_{F_1}+c_1=T_1h'(F_1^a),
\qquad
\frac{|p|^2}{2}+\Phi+V_{F_1}+c_2=T_2h'(F_1^a).
\]
Subtracting the two identities gives
\[
(T_1-T_2)h'(F_1^a)=c_1-c_2.
\]
Hence $h'(F_1^a)$, and therefore $F_1^a$, would be constant on the set where $F_1^a>0$. Since
$F_1^a\in L^1(\R^6)$ and $F_1^a\not\equiv0$ by Proposition \ref{prop-ac-nonzero}, this is impossible.
Therefore $E_1<E_2$.

It remains to prove the limit as $T\downarrow0$. We first record an approximation fact. For every
$\varepsilon>0$ there exists $G_\varepsilon\in\mathcal A$ such that $H(G_\varepsilon)>-\infty$ and
\[
\mathcal E(G_\varepsilon)\leq E_{\min}+\varepsilon.
\]
Indeed, choose $F\in\mathcal A$ with $\mathcal E(F)\leq E_{\min}+\varepsilon/3$. Truncate $F$ to a
large compact set in $(x,p)$ and multiply by the factor that restores mass $M$. For an increasing
sequence of such compact sets, the retained measures increase to $F$ and the restoring factors tend
to one. The weighted linear terms therefore converge by monotone convergence applied separately to
$\Phi_+$ and the bounded negative part of $\Phi$. If $\rho_R$ is the retained spatial marginal, then
\[
 \iint\frac{d\rho_R(x)d\rho_R(y)}{|x-y|}
 \uparrow
 \iint\frac{d\rho_F(x)d\rho_F(y)}{|x-y|},
\]
so the Coulomb energies converge as well. Next convolve the compactly supported phase-space measure
with a product mollifier and restore its mass if needed. The resulting smooth compactly supported
densities have finite entropy. Their kinetic and external energies converge by continuity on the
fixed compact support, while their Coulomb energies converge by the Fourier representation
\eqref{eq-fourier-coulomb} and dominated convergence. Choosing first the truncation and then the
mollification sufficiently fine gives the desired $G_\varepsilon$.

Next we prove an entropy lower bound in terms of total energy. By Lemma \ref{lem-decay}, for every
$\eta>0$ there exists $C_\eta>0$ such that
\[
h(s)\leq \eta\left(\frac{|p|^2}{2}+\Phi_+(x)\right)s+C_\eta s
+2e^{-\frac{|p|+\sqrt{\Phi_+(x)}}2}.
\]
Using the fixed mass constraint, the integrability of the exponential term, the boundedness of
$\Phi_-$, and the nonnegativity of the Coulomb energy, we obtain, for all $F\in\mathcal A$,
\[
-H(F)=\iint h(F^a)\,dpdx\leq \eta\mathcal E(F)+C_\eta.
\]
Equivalently,
\[
H(F)\geq -\eta\mathcal E(F)-C_\eta.
\]
Applying this to $F_T$ gives
\[
\mathcal F_T(F_T)=E_T+TH(F_T)
\geq (1-\eta T)E_T-TC_\eta.
\]

By minimality of $F_T$ and by the choice of $G_\varepsilon$,
\[
(1-\eta T)E_T-TC_\eta
\leq \mathcal F_T(F_T)
\leq \mathcal F_T(G_\varepsilon)
=\mathcal E(G_\varepsilon)+TH(G_\varepsilon)
\leq E_{\min}+\varepsilon+TH(G_\varepsilon).
\]
Since $H(G_\varepsilon)$ is finite and fixed, letting $T\downarrow0$ yields
\[
\limsup_{T\downarrow0}E_T\leq E_{\min}+\varepsilon.
\]
As $\varepsilon>0$ is arbitrary and $E_T\geq E_{\min}$ by definition, we conclude that
\[
\lim_{T\downarrow0}E_T=E_{\min}.
\]
\end{proof}

\subsection{A preliminary low-temperature condensed regime}

By Lemmas \ref{lem-Jmin-gap} and
\ref{lem-ET-limit}, there exists $T_{\rm low}>0$ such that
\[
E_T<J_{\min}
\qquad\text{for every }0<T<T_{\rm low}.
\]
Suppose, for contradiction, that $F_T$ has no singular part for some $0<T<T_{\rm low}$. Then Theorem
\ref{prop-EL-minimizer} shows that $F_T$ is a purely absolutely continuous normal state; hence for
some $c\in\R$ and $V=V_{F_T}$ we have $(T,c,V)\in\mathcal B_M$ and
\[
E_T=J(T,c,V)\geq J_{\min}.
\]
This contradicts $E_T<J_{\min}$. Therefore $F_T$ has a nontrivial singular part for every
$0<T<T_{\rm low}$. Thus the fixed-mass condensed-temperature set contains a nonempty interval.

\subsection{The grand-canonical obstacle problem}
For $T>0$ and $s\geq0$, define the pressure
\[
 P_T(s):=-T\int_{\R^3}\log\left(1-
 e^{-(|p|^2/2+s)/T}\right)\,dp.
\]
Then $P_T\in C^1([0,\infty))$ is strictly decreasing and convex, and
\begin{equation}\label{eq-pressure-derivative}
 P_T'(s)=-R_T(s).
\end{equation}
For a chemical potential $\mu\in\R$, set
\[
 \mathcal K_\mu:=\{V\in\dot H^1(\R^3):
 \Phi+\widetilde V-\mu\geq0\ \text{q.e.}\}
\]
and consider
\begin{equation}\label{eq-grand-obstacle-functional}
 \mathfrak J_{T,\mu}(V):=\frac1{8\pi}\int_{\R^3}|\nabla V|^2\,dx
 +\int_{\R^3}P_T(\Phi+V-\mu)\,dx,
 \qquad V\in\mathcal K_\mu.
\end{equation}
The integral is understood as an extended nonnegative integral.

\begin{lemma}[Grand-canonical obstacle state]\label{lem-grand-obstacle}
For every $T>0$ and $\mu\in\R$, the functional
$\mathfrak J_{T,\mu}$ has a unique minimizer $V_{T,\mu}$. Put
\[
 U_{T,\mu}:=\Phi+\widetilde V_{T,\mu}-\mu.
\]
There is a finite nonnegative Radon measure $\lambda_{T,\mu}$ such that
\begin{equation}\label{eq-grand-obstacle-equation}
 -\frac1{4\pi}\Delta V_{T,\mu}
 =R_T(U_{T,\mu})\,dx+\lambda_{T,\mu},
 \qquad
 U_{T,\mu}=0\quad\lambda_{T,\mu}\text{-a.e.}
\end{equation}
Both measures on the right have finite Coulomb energy. Moreover,
$V_{T,\mu}\geq0$ q.e., and
\[
 \operatorname{supp}\lambda_{T,\mu}
 \subset\{x:\Phi(x)\leq\mu\}.
\]

The phase-space measure
\begin{equation}\label{eq-grand-state}
 \mathfrak F_{T,\mu}:=
 \frac{dxdp}{\exp((|p|^2/2+U_{T,\mu}(x))/T)-1}
 +\lambda_{T,\mu}(dx)\otimes\delta_0(dp)
\end{equation}
is the unique minimizer of the grand potential
\[
 F\longmapsto\mathcal F_T(F)-\mu F(\R^6)
\]
over the union of the admissible classes of all finite masses. In particular,
$\lambda_{T,\mu}=0$ if and only if \eqref{eq-grand-state} is a normal state; in that case it is the
unique canonical minimizer at its own mass.
\end{lemma}

\begin{proof}
The set $\mathcal K_\mu$ is nonempty: a smooth compactly supported function can be chosen above the
continuous compactly supported obstacle $(\mu-\Phi)_+$. It is convex and weakly closed in
$\dot H^1$. Indeed, the usual quasi-everywhere convergence theorem for bounded sequences in
$\dot H^1$, applied after taking convex combinations, preserves the obstacle inequality. Since
$P_T\geq0$, the Dirichlet term makes every minimizing sequence bounded in $\dot H^1$.
Fatou's lemma along a quasi-everywhere convergent sequence gives weak lower semicontinuity. Hence a
minimizer exists. The strict convexity of the Dirichlet term on $\dot H^1$ gives uniqueness.

Replacing the minimizer by its positive part preserves the obstacle constraint, does not increase
the Dirichlet term, and decreases the pressure term. Thus $V_{T,\mu}\geq0$ q.e. Consequently,
\[
 U_{T,\mu}\geq(\Phi-\mu)_+\quad\text{q.e.},
\]
and therefore both $R_T(U_{T,\mu})$ and $P_T(U_{T,\mu})$ are integrable by the standing confinement
assumption and their exponential decay at infinity.

The variational inequality is
\begin{equation}\label{eq-grand-VI}
 \frac1{4\pi}\int_{\R^3}\nabla V_{T,\mu}\cdot
 \nabla(W-V_{T,\mu})\,dx
 -\int_{\R^3}R_T(U_{T,\mu})(W-V_{T,\mu})\,dx\geq0
\end{equation}
for every $W\in\mathcal K_\mu$ for which the displayed terms are finite. Taking nonnegative test increments $W=V_{T,\mu}+\varphi$ for $\varphi \geq0$, $\varphi\in C_c^\infty(\R^3)$, shows that
\[
 \lambda_{T,\mu}:=-\frac1{4\pi}\Delta V_{T,\mu}
 -R_T(U_{T,\mu})\,dx
\]
is a positive distribution, hence a Radon measure. The normal-cone characterization of the
obstacle constraint in \eqref{eq-grand-VI} gives
$U_{T,\mu}=0$ $\lambda_{T,\mu}$-a.e. The measure does not charge sets of Newtonian capacity zero,
because it belongs locally to $H^{-1}$; hence the quasi-everywhere statement is sufficient here.
On the contact set, $V_{T,\mu}=\mu-\Phi\geq0$, so the reaction is supported in the compact set
$\{\Phi\leq\mu\}$. It is therefore finite. Equation \eqref{eq-grand-obstacle-equation} and the
$\dot H^1$ bound show that the total measure on its right has finite Coulomb energy by Proposition \ref{prop-coulomb-energy}, and then so do
its two nonnegative summands.

Let $f_{T,\mu}$ denote the density in \eqref{eq-grand-state}. Since
$U_{T,\mu}\geq(\Phi-\mu)_+$, direct integration of the Bose factor against
$1+|p|^2/2+\Phi_+(x)$, and likewise of $h(f_{T,\mu})$, gives finite mass, moments, and entropy.
Together with the compact support of the reaction and its finite Coulomb energy, this shows that
$\mathfrak F_{T,\mu}$ is admissible. For any admissible $F$, the Euler identity
\[
 \frac{|p|^2}{2}+U_{T,\mu}=T h'(f_{T,\mu})
\]
and the Coulomb polarization formula give
\begin{align}\label{eq-grand-expansion}
 &\mathcal F_T(F)-\mu F(\R^6)
 -\bigl(\mathcal F_T(\mathfrak F_{T,\mu})
       -\mu\mathfrak F_{T,\mu}(\R^6)\bigr) \notag\\
 &\quad=T d(F^a,f_{T,\mu})
 +\iint_{\R^6}\left(\frac{|p|^2}{2}+U_{T,\mu}(x)\right)dF^s(x,p)
 +\frac1{8\pi}\|\nabla(V_F-V_{T,\mu})\|_2^2,
\end{align}
where $d(\cdot,\cdot)$ is defined in Definition \ref{D.d.def}.
The singular part of the reference state contributes zero because it is supported where $p=0$ and
$U_{T,\mu}=0$. Every term on the right is nonnegative. Equality fixes the regular density, the
spatial marginal through the potential, and then the singular phase-space measure through the zero
set of $|p|^2/2+U_{T,\mu}$. This proves uniqueness and all remaining assertions.
\end{proof}

We use the following zero-set form of Kato's inequality. If $w\geq0$ quasi-everywhere,
$w\in W^{1,1}_{\rm loc}$, and $\Delta w$ is a Radon measure diffuse with respect to Newtonian
capacity, then
\begin{equation}\label{eq-kato-zero-set}
 (\Delta w)\mathbin{\vrule height 1.4ex depth -0.3ex width 0.07ex
 \vrule height 0.07ex depth -0.02ex width 0.8ex}_{\{\widetilde w=0\}}\geq0.
\end{equation}
Here and below the zero set is defined using the precise representative. This is the
measure-valued Kato inequality; see \cite[Corollary~1.3]{BrezisPonce2004}.

\subsection{Comparison in temperature and chemical potential}
\begin{lemma}[Temperature comparison]\label{lem-temperature-comparison}
Let $0<T_1<T_2$ and fix $\mu\in\R$. Write
\[
 V_i=V_{T_i,\mu},\quad U_i=\Phi+V_i-\mu,
 \quad\lambda_i=\lambda_{T_i,\mu}.
\]
Then $V_1\leq V_2$ q.e. If $\lambda_1=0$, then $\lambda_2=0$ and
\begin{equation}\label{eq-temperature-mass-comparison}
 \int_{\R^3}R_{T_1}(U_1)\,dx
 \leq\int_{\R^3}R_{T_2}(U_2)\,dx.
\end{equation}
\end{lemma}

\begin{proof}
The series representation
\[
 R_T(s)=(2\pi T)^{3/2}\sum_{k=1}^\infty k^{-3/2}e^{-ks/T}
\]
shows that $R_T(s)$ is strictly increasing in $T$ for every $s\geq0$. Let
$q=(V_1-V_2)_+$. The lattice competitors $V_1-q=\min\{V_1,V_2\}$ and
$V_2+q=\max\{V_1,V_2\}$ belong to $\mathcal K_\mu$. Adding their two variational inequalities 
\[
-\frac1{4\pi}\int_{\R^3} \nabla V_1 \cdot \nabla q\,dx + \int_{\R^3}R_{T_1}(U_1)q\,dx \geq 0 , \quad \mbox{and} \quad \frac1{4\pi}\int_{\R^3} \nabla V_2 \cdot \nabla q\,dx - \int_{\R^3}R_{T_2}(U_2)q\,dx \geq 0
\]
gives
\[
 \frac1{4\pi}\int_{\R^3}|\nabla q|^2\,dx
 \leq\int_{\{q>0\}}\bigl(R_{T_1}(U_1)-R_{T_2}(U_2)\bigr)q\,dx\leq0,
\]
because $U_1>U_2$ on $\{q>0\}$. Hence $V_1\leq V_2$ q.e.

Assume now that $\lambda_1=0$ and set $w=V_2-V_1\geq0$. Subtracting the obstacle equations yields
\begin{equation}\label{eq-temperature-difference}
 -\frac1{4\pi}\Delta w
 =\bigl(R_{T_2}(U_2)-R_{T_1}(U_1)\bigr)\,dx+\lambda_2.
\end{equation}
Since $U_2-U_1=w$, the contact condition for $\lambda_2$ implies
\[
 \lambda_2\bigl(\{\widetilde w>0\}\bigr)=0.
\]
On the zero set $Z=\{\widetilde w=0\}$ one has $U_1=U_2$, and therefore the right-hand side of
\eqref{eq-temperature-difference}, restricted to $Z$, is a nonnegative measure. On the other hand,
\eqref{eq-kato-zero-set} gives
\[
 \left(-\Delta w\right)\mathbin{\vrule height 1.4ex depth -0.3ex width 0.07ex
 \vrule height 0.07ex depth -0.02ex width 0.8ex}_{Z}\leq0.
\]
Thus both sides restricted to $Z$ vanish, and in particular $\lambda_2=0$.

The two potentials are now Newtonian potentials of the integrable bounded densities
$\rho_i=R_{T_i}(U_i)$. Newton's spherical-average identity gives, for every $R>0$,
\[
 \frac{R}{4\pi}\int_{\mathbb S^2}w(R\omega)\,d\omega
 =\int_{\R^3}\frac{R}{\max\{R,|y|\}}\,(\rho_2-\rho_1)dy.
\]
The left-hand side is nonnegative. Letting $R\to\infty$ and using dominated convergence on the
right proves \eqref{eq-temperature-mass-comparison}.
\end{proof}

\begin{lemma}[Chemical-potential comparison]\label{lem-chemical-comparison}
Fix $T>0$. If $\lambda_{T,\mu}=0$ for some $\mu\in\R$, then
$\lambda_{T,\nu}=0$ for every $\nu<\mu$. The corresponding normal masses
\[
 m_T(\nu):=\int_{\R^3}R_T(\Phi+V_{T,\nu}-\nu)\,dx,
 \qquad \nu\leq\mu,
\]
are nondecreasing and continuous in $\nu$, and
\begin{equation}\label{eq-mass-minus-infinity}
 \lim_{\nu\to-\infty}m_T(\nu)=0.
\end{equation}
Consequently, every mass in $(0,m_T(\mu)]$ is attained by a normal state at temperature $T$.
\end{lemma}

\begin{proof}
Fix $\nu<\mu$ and abbreviate $V_\alpha=V_{T,\alpha}$,
$U_\alpha=\Phi+V_\alpha-\alpha$. The same lattice argument as above, now using the nested
obstacles, gives $V_\nu\leq V_\mu$. Set
\[
 Q:=U_\nu-U_\mu=V_\nu-V_\mu+\mu-\nu.
\]
Since $Q_-\leq(V_\mu-V_\nu)_+$, its negative part is an admissible $\dot H^1$ test function.
Subtracting the equations and testing by $Q_-$ gives
\[
 -\frac1{4\pi}\int_{\R^3}|\nabla Q_-|^2\,dx
 =\int_{\R^3}\bigl(R_T(U_\nu)-R_T(U_\mu)\bigr)Q_-\,dx
 +\int_{\R^3}Q_-\,d\lambda_{T,\nu}\geq0.
\]
Indeed, $R_T(U_\nu)>R_T(U_\mu)$ wherever $Q<0$. Hence $Q\geq0$ q.e. The reaction
$\lambda_{T,\nu}$ is supported on $\{U_\nu=0\}$, and there $U_\nu=U_\mu+Q$ forces
$U_\mu=Q=0$. Restricting
\[
 -\frac1{4\pi}\Delta Q
 =\bigl(R_T(U_\nu)-R_T(U_\mu)\bigr)\,dx+\lambda_{T,\nu}
\]
to $\{\widetilde Q=0\}$ and applying \eqref{eq-kato-zero-set} yields
$\lambda_{T,\nu}=0$.

Since $V_\mu-V_\nu\geq0$ and both states are normal, the same spherical-average argument as in
Lemma~\ref{lem-temperature-comparison} shows that $m_T(\nu)\leq m_T(\mu)$. Applying this to any
two parameters below $\mu$ proves monotonicity.

We record the continuity argument because it is also needed at the endpoint. Let
$\nu_n\to\nu\leq\mu$, with $\nu_n\leq\mu$. The bounds
\[
 0\leq R_T(U_{\nu_n})
 \leq R_T((\Phi-\nu_n)_+)
 \leq R_T((\Phi-\mu')_+)
\]
hold for a fixed $\mu'>\mu$ and all large $n$. The last function belongs to
$L^1\cap L^\infty$. Hence the masses, Newton potentials, and $\dot H^1$ norms are uniformly
bounded. After extraction, $V_{\nu_n}$ converges weakly in $\dot H^1$, almost everywhere, and in
$L^q_{\rm loc}$ for $q<6$. Dominated convergence in the Poisson equation identifies the limit with
$V_{T,\nu}$ by uniqueness in Lemma~\ref{lem-grand-obstacle}; it also gives
$m_T(\nu_n)\to m_T(\nu)$. Finally,
\[
 0\leq m_T(\nu)\leq\int_{\R^3}R_T(\Phi-\nu)\,dx
\]
for all sufficiently negative $\nu$. The right-hand side tends to zero by dominated convergence and
the confinement assumption, proving \eqref{eq-mass-minus-infinity}. The intermediate value theorem
now gives the last assertion.
\end{proof}

\begin{corollary}[Monotonicity of the canonical normal phase]\label{cor-normal-monotone}
If the canonical minimizer of mass $M$ is normal at temperature $T_1$, then it is normal at every
$T_2>T_1$. If the canonical minimizer of mass $M_2$ is normal at temperature $T$, then it is normal
at every mass $0<M_1<M_2$.
\end{corollary}

\begin{proof}
Let $\mu$ be the chemical potential of the first normal minimizer. By
Lemma~\ref{lem-grand-obstacle}, $\lambda_{T_1,\mu}=0$ and $m_{T_1}(\mu)=M$.
Lemma~\ref{lem-temperature-comparison} gives
$\lambda_{T_2,\mu}=0$ and $m_{T_2}(\mu)\geq M$. Lemma~\ref{lem-chemical-comparison} then supplies a
$\nu\leq\mu$ with $m_{T_2}(\nu)=M$. The resulting normal state is the unique canonical minimizer by
Lemma~\ref{lem-grand-obstacle}. This proves temperature monotonicity. The same chemical-potential
comparison, without changing $T$, proves mass monotonicity.
\end{proof}

\subsection{Closedness of the normal phase and proof of the sharp transition}
\begin{lemma}[Closed normal graph]\label{lem-normal-closed}
Let $T_n\to T>0$ and $M_n\to M>0$. If the canonical minimizer $F_n$ at $(T_n,M_n)$ is normal for
every $n$, then the canonical minimizer at $(T,M)$ is normal.
\end{lemma}

\begin{proof}
Choose a smooth compactly supported density $G$ of mass $M$ and put
$G_n=(M_n/M)G$. Minimality of $F_n$, together with the entropy estimate in
Lemma~\ref{lem-decay}, gives the uniform bound
\[
 \sup_n\left\{
 \iint_{\R^6}\left(\frac{|p|^2}{2}+\Phi_+(x)\right)dF_n
 +\|\nabla V_{F_n}\|_2^2+|H(F_n)|\right\}<\infty.
\]
This is exactly the coercivity argument used in the proof of Proposition~\ref{prop-existence-minimizer},
and it is uniform because $T_n$ stays in a compact subset of $(0,\infty)$ and $M_n$ stays bounded.
Thus, after extraction, $F_n\stackrel{*}{\rightharpoonup}F$ as finite measures, with no loss of mass
and with $F(\R^6)=M$.

Because $F_n$ is normal, Theorem~\ref{prop-EL-minimizer} gives
\[
 0\leq F_n(x,p) =\frac{1}{\exp((|p|^2/2+U_n)/T_n)-1} 
 \leq\frac1{\exp(|p|^2/(2T_+))-1}
\]
for some $T_+>\sup_nT_n$. The function on the right is integrable in $p$. On every bounded spatial
set it therefore gives a common integrable dominating density, and the weak-* limit $F$ is
absolutely continuous. In particular, $F$ has no singular part.

Lemma~\ref{lsc} and the uniform entropy bound give
\[
 \mathcal F_T(F)\leq \liminf_{n\to\infty}\mathcal F_{T}(F_n) =\liminf_{n\to\infty}\mathcal F_{T_n}(F_n).
\]
For an arbitrary $A\in\mathcal A_M$, set $A_n=(M_n/M)A$. Directly from the definition of the
functional, the Coulomb polarization formula, and continuity of the integral of $j$ under a scalar
factor tending to one (use $h(\alpha s)\leq C(h(s)+s)$ for $1/2\leq\alpha\leq2$),
\[
 \mathcal F_{T_n}(A_n)\longrightarrow\mathcal F_T(A).
\]
The minimality inequalities
$\mathcal F_{T_n}(F_n)\leq\mathcal F_{T_n}(A_n)$ therefore imply
$\mathcal F_T(F)\leq\mathcal F_T(A)$. Hence $F$ is the unique minimizer at $(T,M)$, and it is
normal.
\end{proof}

\begin{proof}[Proof of Theorem~\ref{thm-T}]
For fixed $M>0$, let
\[
 \mathfrak N_M:=\{T>0:\rho_T^s=0\}.
\]
The preliminary high-temperature construction shows that $\mathfrak N_M$ is nonempty, whereas the
preliminary low-temperature comparison shows that its complement contains an interval $(0,T_{\rm
low})$. Corollary~\ref{cor-normal-monotone} shows that $\mathfrak N_M$ is an upper interval, and
Lemma~\ref{lem-normal-closed} shows that it is closed in $(0,\infty)$. Consequently there is a
unique $T_\sharp(M)\in(0,\infty)$ such that
\[
 \mathfrak N_M=[T_\sharp(M),\infty).
\]
This is exactly the assertion of Theorem~\ref{thm-T}.
\end{proof}

\section{Sharp critical mass and harmonic-trap dichotomy}\label{sec-mass-transition}

We now fix the temperature and vary the total mass. The comparison principle from the preceding
section shows that the normal masses form a closed initial interval. We identify its endpoint with
the full normal-branch capacity and then determine when that endpoint is finite.

\subsection{A preliminary small-mass normal regime}
For fixed $T>0$ define
\[
M^*(T)=\sup_{c\in[-\Phi_{\min},\infty)}M_{T,c}.
\]
Since $M_{T,c}>0$ for every admissible $c$, we have $M^*(T)>0$. If $0<M<M^*(T)$, then there exists
$c_1\geq-\Phi_{\min}$ with $M<M_{T,c_1}$. By Lemma \ref{lem-continuity-MTc} and
$M_{T,c}\to0$ as $c\to\infty$, there exists $c_0\geq c_1$ such that $M=M_{T,c_0}$. Lemma
\ref{lem-expansion} then implies that the unique minimizer of mass $M$ is $F_{T,c_0}$ and hence has
no singular part. Thus the normal-mass set contains a nonempty interval.

\subsection{Full normal branch and an abstract large-mass criterion}
For $T>0$ set
\[
R_T(s)=\int_{\R^3}\frac{1}{\exp\left(\frac1T(\frac{|p|^2}{2}+s)\right)-1}\,dp,
\qquad s\geq0.
\]
Then $R_T$ is strictly decreasing and $R_T(s)$ decays exponentially as $s\to\infty$. Moreover,
\[
 R_T(0)=(2\pi T)^{3/2}\zeta(3/2),
\]
where $\zeta(s)=\sum_{n=1}^{\infty}n^{-s}$ is the Riemann zeta function. Indeed, using spherical
coordinates and the change of variables $r^2/(2T)=s$,
\[
\begin{aligned}
R_T(0)
&=4\pi\int_0^\infty\frac{r^2}{e^{r^2/(2T)}-1}\,dr
=2\pi(2T)^{3/2}\int_0^\infty\frac{s^{1/2}}{e^s-1}\,ds \\
&=2\pi(2T)^{3/2}\Gamma(3/2)\zeta(3/2)
=(2\pi T)^{3/2}\zeta(3/2).
\end{aligned}
\]

Define the full normal branch by
\[
\mathcal N_T=\left\{(\mu,V)~\middle|~
\begin{aligned}
&\mu\in\R,\quad V\in\dot H^1(\R^3),\quad
U:=\Phi+\widetilde V-\mu\geq0\ \text{q.e.},\\
&-\Delta V=4\pi R_T(U)\quad\text{in }\mathcal D'(\R^3)
\end{aligned}\right\},
\]
For $(\mu,V)\in\mathcal N_T$, set
\[
M_T(\mu,V)=\int_{\R^3}R_T(\Phi+V-\mu)\,dx
\]
and
\[
\widehat M^*(T)=\sup_{(\mu,V)\in\mathcal N_T}M_T(\mu,V).
\]
This is the normal-branch capacity; below we prove that it equals the sharp critical
mass. Notice that
it is not enough to restrict to the sub-branch $c\geq-\Phi_{\min}$, because a general normal state only
satisfies $\Phi+V+c\geq0$.

\begin{lemma}[Regularity of the normal branch]\label{lem-normal-branch-regularity}
If $(\mu,V)\in\mathcal N_T$, then
\[
 \rho:=R_T(\Phi+V-\mu)\in L^1(\R^3)\cap L^\infty(\R^3),
 \qquad
 V(x)=\int_{\R^3}\frac{\rho(y)}{|x-y|}\,dy.
\]
In particular, $V$ has a nonnegative continuous representative, $U=\Phi+V-\mu$ is continuous and
nonnegative everywhere, and the associated phase-space Bose density
\[
 F_N(x,p):=\frac1{\exp\left(\frac1T(\frac{|p|^2}{2}+U(x))\right)-1}
\]
belongs to $\mathcal A_{M_T(\mu,V)}$ and has finite entropy.
\end{lemma}
\begin{proof}
The weak maximum principle in $\dot H^1$ gives $V\geq0$ q.e. To see this without first assuming
that the right-hand side is integrable, let $\chi_R$ be a cutoff that equals one on $B_R$, vanishes
outside $B_{2R}$, and satisfies $|\nabla\chi_R|\leq C/R$. Testing by
$\chi_R^2V_-$, where $V_-=\max\{-V,0\}$, gives
\[
 \int_{\R^3}\chi_R^2|\nabla V_-|^2
 \leq C\int_{B_{2R}\setminus B_R}V_-^2|\nabla\chi_R|^2.
\]
By H\"older's inequality, the right-hand side is bounded by
\[
 C\left(\int_{B_{2R}\setminus B_R}|V_-|^6\,dx\right)^{1/3},
\]
which tends to zero because $V_-\in L^6$. Thus $\nabla V_-=0$, and the $L^6$ condition implies
$V_-=0$. Combining $V\geq0$ with $U\geq0$ gives
\[
 U\geq(\Phi-\mu)_+\qquad\text{q.e.}
\]
Since $R_T$ is decreasing,
\[
 0\leq\rho\leq R_T((\Phi-\mu)_+)\leq R_T(0).
\]
The right-hand side is integrable: it is bounded on the compact sublevel sets of $\Phi$, and for
large $\Phi$ it is bounded by $C_{T,\mu}e^{-\Phi/(2T)}$, which is integrable under the standing
confinement assumption. Thus $\rho\in L^1\cap L^\infty$, and in particular
$\rho\in L^{6/5}$.

The Newtonian potential $W=|x|^{-1}*\rho$ belongs to $\dot H^1$ by the
Hardy--Littlewood--Sobolev inequality and satisfies $-\Delta W=4\pi\rho$. Hence $V-W$ is harmonic
in distributions and belongs to $\dot H^1$; its gradient vanishes, and its $L^6$ representative
forces the additive constant to be zero. Thus $V=W$. The Newtonian potential of an
$L^1\cap L^\infty$ density is continuous and nonnegative. Consequently $U$ is continuous, and its
quasi-everywhere nonnegativity improves to pointwise nonnegativity. Finally,
$M_T(\mu,V)=\|\rho\|_1<\infty$.

The same comparison with $(\Phi-\mu)_+$, now after integrating the Bose factor against
$1+|p|^2/2+\Phi_+(x)$, gives
\[
 \iint_{\R^6}\left(1+\frac{|p|^2}{2}+\Phi_+(x)\right)F_N(x,p)\,dpdx<\infty.
\]
Indeed, the momentum integrals are bounded on spatial sublevel sets and decay exponentially in
$\Phi$ outside them. The analogous estimate for $h(F_N)$ gives finite entropy. Together with
$V\in\dot H^1$, these estimates prove $F_N\,dxdp\in\mathcal A_{M_T(\mu,V)}$.
\end{proof}

\begin{prop}\label{prop-normal-minimizer}
Let $T>0$, $c\in\R$, and $V\in\dot H^1(\R^3)$. Set
\[
U:=\Phi+V+c
\]
and assume that $U\geq0$ q.e. and
\[
-\Delta V=4\pi R_T(U)
\]
in the sense of distributions. Define
\[
F_N(x,p)=\frac{1}{\exp\left(\frac1T\left(\frac{|p|^2}{2}+U(x)\right)\right)-1}.
\]
Then $F_N$ has finite mass $M=M_T(-c,V)$ and is the unique minimizer of
\eqref{mp-boson} in $\mathcal A_M$.
\end{prop}
\begin{proof}
Lemma~\ref{lem-normal-branch-regularity}, applied with $\mu=-c$, shows that
$F_N\,dxdp\in\mathcal A_M$ and that $V=V_{F_N}$. Let $F\in\mathcal A_M$ have the same mass $M$.
Since
\[
\frac{|p|^2}{2}+U=T h'(F_N),
\]
the same expansion as in Lemma \ref{lem-expansion}, with the chemical
potential $c$ cancelling by the mass constraint, gives
\[
\begin{aligned}
\mathcal F_T(F)-\mathcal F_T(F_N)
&=\iint_{\R^6}\left(\frac{|p|^2}{2}+U(x)\right)dF^s
+T d(F^a,F_N) \\
&\quad +\frac{1}{8\pi}\|\nabla(V_F-V)\|_2^2.
\end{aligned}
\]
Every term on the right-hand side is nonnegative: the first one by $U\geq0$ q.e. and the fact that
finite-energy measures do not charge capacity-zero sets, the second by strict convexity of $-h$, and
the third by positivity of the Coulomb energy. Therefore $F_N$ is a minimizer. If equality holds, then
$d(F^a,F_N)=0$, $\nabla V_F=\nabla V$, and the first term vanishes. Thus $F^a=F_N$ a.e. and the mass
constraint forces $F^s=0$. Hence $F=F_N$, proving uniqueness.
\end{proof}

\begin{prop}\label{prop-cond-threshold}
If $M>\widehat M^*(T)$, then the minimizer of \eqref{mp-boson} has a nontrivial singular part.
\end{prop}
\begin{proof}
Suppose that the minimizer $F_0$ has no singular part. By Theorem \ref{prop-EL-minimizer},
there is a constant $c\in\R$ such that, with $\mu=-c$,
\[
F_0(x,p)=\frac{1}{\exp\left(\frac1T(\frac{|p|^2}{2}+\Phi(x)+V_{F_0}(x)-\mu)\right)-1}
\]
and $U=\Phi+V_{F_0}-\mu\geq0$. Thus $(\mu,V_{F_0})\in\mathcal N_T$, and
\[
M=\int_{\R^3}R_T(U)\,dx\leq\widehat M^*(T),
\]
contradicting the assumption.
\end{proof}

\begin{prop}\label{prop-obstacle-finite}
Assume that $\Phi$ attains its minimum at $x_0$. Define
\[
K_T(\mu)=\int_{\R^3}\frac{R_T((\Phi(y)-\mu)_+)}{|x_0-y|}\,dy.
\]
If
\[
K_T(\mu)=o(\mu)\qquad\text{as }\mu\to\infty,
\]
then $\widehat M^*(T)<\infty$.
\end{prop}
\begin{proof}
Let $(\mu,V)\in\mathcal N_T$ and put $U=\Phi+V-\mu$. Since $U\geq0$ and $V\geq0$,
\[
U(x)\geq(\Phi(x)-\mu)_+.
\]
Because $R_T$ is decreasing,
\[
\rho(x):=R_T(U(x))\leq R_T((\Phi(x)-\mu)_+).
\]
At the minimum point $x_0$ of $\Phi$, the obstacle condition gives
\[
\mu-\Phi_{\min}\leq V(x_0).
\]
Using the Newton representation of $V$ and the previous density bound, we obtain
\[
\mu-\Phi_{\min}
\leq\int_{\R^3}\frac{\rho(y)}{|x_0-y|}\,dy
\leq K_T(\mu).
\]
If $K_T(\mu)=o(\mu)$, this inequality cannot hold for arbitrarily large $\mu$. Thus all normal branch
parameters satisfy $\mu\leq\mu_*(T)$ for some finite $\mu_*(T)$.

Consequently,
\[
M_T(\mu,V)=\int_{\R^3}\rho(x)\,dx
\leq\int_{\R^3}R_T((\Phi(x)-\mu_*(T))_+)\,dx<\infty,
\]
where the last integral is finite by the confining assumption on $\Phi$ and the exponential decay of
$R_T$. Taking the supremum over $(\mu,V)\in\mathcal N_T$ gives the result.
\end{proof}

\subsection{Superquadratic confinement}
\begin{corollary}\label{cor-superquadratic}
If
\[
\Phi(x)\geq a|x|^q-b\qquad\text{for large }|x|
\]
with $q>2$, then $\widehat M^*(T)<\infty$ for every fixed $T>0$.
\end{corollary}
\begin{proof}
For large $\mu$, decompose space into
\[
 A_0:=\{\Phi\leq\mu+1\},\qquad
 A_k:=\{\mu+k<\Phi\leq\mu+k+1\},\quad k\geq1.
\]
The growth assumption implies
\[
 A_k\subset B_{C(1+\mu+k)^{1/q}}(0).
\]
Since the integral of $|x_0-y|^{-1}$ over a ball of radius $R$ is bounded by
$C(1+R^2)$, we have
\[
 \int_{A_k}\frac{dy}{|x_0-y|}
 \leq C(1+\mu+k)^{2/q}.
\]
On $A_0$ use $R_T\leq R_T(0)$; on $A_k$, $k\geq1$, use
$R_T((\Phi-\mu)_+)\leq C_Te^{-k/(2T)}$. It follows that
\[
\begin{aligned}
 K_T(\mu)
 &\leq C_T(1+\mu)^{2/q}
 +C_T\sum_{k=1}^\infty e^{-k/(2T)}(1+\mu+k)^{2/q}\\
 &\leq C_T(1+\mu)^{2/q}.
\end{aligned}
\]
Since $2/q<1$, $K_T(\mu)=o(\mu)$ as $\mu\to\infty$. Proposition \ref{prop-obstacle-finite} gives
$\widehat M^*(T)<\infty$.
\end{proof}

\subsection{Harmonic trap dichotomy}
\begin{prop}\label{prop-harmonic-strong}
Assume
\[
\Phi(x)=\Phi(0)+\frac{\omega^2}{2}|x|^2
\]
and restrict attention to radial normal states. If
\[
3\omega^2>4\pi R_T(0),
\]
then the radial normal branch has uniformly bounded mass.
\end{prop}
\begin{proof}
Let $(\mu,V)\in\mathcal N_T$ be radial and set
\[
U(r)=\Phi(r)+V(r)-\mu,
\qquad \rho(r)=R_T(U(r)).
\]
Let
\[
m(r)=4\pi\int_0^r\rho(s)s^2\,ds.
\]
Then $V'(r)=-m(r)/r^2$, and hence
\[
U'(r)=\omega^2r-\frac{m(r)}{r^2}.
\]
Therefore
\[
(r^2U'(r))'=r^2\left(3\omega^2-4\pi R_T(U(r))\right).
\]
Since $U\geq0$ and $R_T$ is decreasing,
\[
R_T(U(r))\leq R_T(0).
\]
Let $\delta_T=3\omega^2-4\pi R_T(0)>0$. Then
\[
(r^2U'(r))'\geq\delta_T r^2.
\]
Using regularity at the origin and integrating twice gives
\[
U(r)\geq U(0)+\frac{\delta_T}{6}r^2\geq\frac{\delta_T}{6}r^2.
\]
Thus
\[
\rho(r)\leq R_T\left(\frac{\delta_T}{6}r^2\right),
\]
and consequently
\[
M_T(\mu,V)
=4\pi\int_0^\infty\rho(r)r^2\,dr
\leq4\pi\int_0^\infty R_T\left(\frac{\delta_T}{6}r^2\right)r^2\,dr<\infty.
\]
The upper bound is independent of $(\mu,V)$.
\end{proof}

\begin{prop}\label{prop-harmonic-weak}
Assume
\[
\Phi(x)=\Phi(0)+\frac{\omega^2}{2}|x|^2.
\]
If
\[
3\omega^2\leq4\pi R_T(0),
\]
then, for every $M>0$, there exists a radial purely absolutely continuous normal state of mass $M$.
Consequently, the minimizer of \eqref{mp-boson} with mass $M$ has no singular part.
\end{prop}

\begin{proof}
Let
\[
f(s)=3\omega^2-4\pi R_T(s).
\]
If $3\omega^2<4\pi R_T(0)$, choose the unique $\alpha>0$ such that $f(\alpha)=0$. In the critical
case $3\omega^2=4\pi R_T(0)$, set $\alpha=0$. For $a>\alpha$, solve
\[
U_a''+\frac2rU_a'=f(U_a),\qquad U_a(0)=a,\qquad U_a'(0)=0.
\]
Since $R_T$ is strictly decreasing, $f$ is strictly increasing, and $f(a)>0$. As long as the solution
exists and $U_a\geq a$, we have
\[
(r^2U_a')'=r^2f(U_a)\geq r^2f(a).
\]
Thus $U_a'(r)\geq f(a)r/3$ and
\[
U_a(r)\geq a+\frac{f(a)}6r^2.
\]
Moreover $f(U)\leq3\omega^2$, so the same integrated identity gives
$U_a'(r)\leq\omega^2r$ and $U_a(r)\leq a+\omega^2r^2/2$ on every interval of existence. This
precludes finite-radius blow-up. In particular $U_a$ is global, remains positive, and tends to
$+\infty$. Hence
\[
\rho_a(r):=R_T(U_a(r))
\]
is integrable and decays exponentially. Since
\[
\Delta U_a=3\omega^2-4\pi R_T(U_a),\qquad \Delta\Phi=3\omega^2,
\]
we have
\[
-\Delta(U_a-\Phi)=4\pi\rho_a.
\]
The radial derivative of $U_a-\Phi$ is $-m_a(r)/r^2$, where
$m_a(r)=4\pi\int_0^r\rho_a(s)s^2\,ds$, and therefore $U_a-\Phi$ has a finite limit at infinity.
Thus
\[
U_a=\Phi+V_a-\mu_a
\]
for the Newton potential $V_a$ of $\rho_a$ and for a suitable constant $\mu_a$. Hence
$(\mu_a,V_a)\in\mathcal N_T$. Moreover, $\rho_a\in L^1\cap L^\infty$, so
$V_a\in\dot H^1$. The quadratic lower bound on $U_a$, together with the exponential decay in
$p$, also shows that the corresponding phase-space Bose density has finite kinetic and external
energies and finite entropy. It therefore belongs to its mass-constrained admissible class.

Set
\[
M(a):=4\pi\int_0^\infty R_T(U_a(r))r^2\,dr.
\]
Continuous dependence for the ODE and dominated convergence on compact intervals, together with the
Gaussian tail obtained above, imply that $M(a)$ is continuous on $(\alpha,\infty)$. Moreover,
$M(a)\to0$ as $a\to\infty$, because the lower bound $U_a(r)\geq a+f(a)r^2/6$ and the exponential
decay of $R_T$ force the density to vanish in $L^1$.

On the other hand, as $a\downarrow\alpha$, continuous dependence gives $U_a\to\alpha$ uniformly on
every bounded interval. Hence for every $R>0$ and every small $\varepsilon>0$, choosing
$a>\alpha$ sufficiently close to $\alpha$ yields
\[
U_a(r)\leq\alpha+\varepsilon,\qquad 0\leq r\leq R.
\]
Therefore
\[
M(a)\geq4\pi\int_0^R R_T(U_a(r))r^2\,dr
\geq \frac{4\pi}{3}R_T(\alpha+\varepsilon)R^3.
\]
Given $L>0$, first choose $R$ and then $\varepsilon$ so that the last lower bound exceeds $L$.
Continuous dependence then shows that the bound holds for every $a>\alpha$ sufficiently close to
$\alpha$. Hence
\[
 \lim_{a\downarrow\alpha}M(a)=\infty.
\]
Together with continuity and $\lim_{a\to\infty}M(a)=0$, this shows by the intermediate value
theorem that the range of $M(a)$ contains $(0,\infty)$. Thus for every prescribed mass $M>0$ there
is $a>\alpha$ such that $M(a)=M$.

The corresponding Bose distribution is a purely absolutely continuous normal state. By Proposition
\ref{prop-normal-minimizer}, it is the unique minimizer at mass $M$. Hence the minimizer has no
singular part.
\end{proof}

\subsection{Proof of the sharp fixed-temperature transition}
\begin{proof}[Proof of Theorem~\ref{thm-M}]
Fix $T>0$ and let
\[
 \mathfrak I_T:=\{M>0:\text{the canonical minimizer at mass $M$ is normal}\}.
\]
The preliminary small-mass construction shows that $\mathfrak I_T$ is nonempty.
Corollary~\ref{cor-normal-monotone} shows that it is an initial interval, and
Lemma~\ref{lem-normal-closed}, applied with fixed temperature, shows that it is closed in
$(0,\infty)$. Therefore there is a unique number $M_\sharp(T)\in(0,\infty]$ such that
\[
 \mathfrak I_T=(0,M_\sharp(T)].
\]
When $M_\sharp(T)=\infty$, the right-hand side is understood as $(0,\infty)$. This proves the exact
normal/condensed alternative in part~(a), including normality at a finite endpoint.

Every $(\mu,V)\in\mathcal N_T$ produces, by Proposition~\ref{prop-normal-minimizer}, a canonical
normal minimizer of mass $M_T(\mu,V)$. Conversely, the Euler--Lagrange equation of every canonical
normal minimizer places its chemical potential and potential in $\mathcal N_T$. Hence
\begin{equation}\label{eq-sharp-mass-normal-capacity}
 M_\sharp(T)=\widehat M^*(T).
\end{equation}

Under superquadratic confinement, Corollary~\ref{cor-superquadratic} and
\eqref{eq-sharp-mass-normal-capacity} give $M_\sharp(T)<\infty$. In the harmonic case satisfying
$3\omega^2>4\pi R_T(0)$, every canonical minimizer is radial by
Theorem~\ref{prop-existence-minimizer}. A normal-branch state is the canonical minimizer at its own
mass, so it is radial as well. Proposition~\ref{prop-harmonic-strong} therefore bounds all masses in
$\mathcal N_T$, and again $M_\sharp(T)<\infty$. This proves part~(b).

If $3\omega^2\leq4\pi R_T(0)$, Proposition~\ref{prop-harmonic-weak} supplies a normal state at every
mass. Thus $M_\sharp(T)=\infty$, proving part~(c).

Finally, Corollary~\ref{cor-normal-monotone} implies
\[
 T_1<T_2\quad\Longrightarrow\quad M_\sharp(T_1)\leq M_\sharp(T_2).
\]
For each fixed $M>0$, Theorem~\ref{thm-T} and the definition of $M_\sharp$ give
\[
 \{T>0:M\leq M_\sharp(T)\}=[T_\sharp(M),\infty).
\]
Taking the minimum proves part~(d).
\end{proof}

\section{Abstract variational stability}\label{sec-dyn-stab}

This section proves Proposition~\ref{thm-dyn-stab}. The point is that the variational problem already
contains the correct Lyapunov functional. Therefore no new coercivity mechanism is needed beyond the
compactness of minimizing sequences and the uniqueness of the minimizer.

\subsection{Formal kinetic motivations and admissible free-energy curves}
We first record the kinetic equations that motivate the abstract curve class used in
Proposition~\ref{thm-dyn-stab}. In all cases the self-consistent potential is given by
\[
 -\Delta V_F=4\pi\rho_F,
 \qquad
 \rho_F(t,x)=\int_{\R^3}F(t,x,p)\,dp
\]
whenever $F$ is absolutely continuous. For measure-valued solutions, $\rho_F$ denotes the spatial
projection of $F$.

The conservative Vlasov--Poisson dynamics with external trap is
\begin{equation}\label{eq-vp}
 \partial_t F+p\cdot\nabla_xF-
 \nabla_x(\Phi+V_F)\cdot\nabla_pF=0,
 \qquad -\Delta V_F=4\pi\rho_F.
\end{equation}
For sufficiently regular solutions, the mass, total energy, and entropy are conserved. Hence
$\mathcal F_T$ is conserved. For weak or measure-valued solutions, Proposition~\ref{thm-dyn-stab}
requires only the corresponding free-energy inequality on the interval of existence.

The spatially inhomogeneous bosonic Boltzmann--Poisson equation is
\begin{equation}\label{eq-boltzmann-poisson}
 \partial_t F+p\cdot\nabla_xF-
 \nabla_x(\Phi+V_F)\cdot\nabla_pF=Q_{\rm BE}(F),
 \qquad -\Delta V_F=4\pi\rho_F.
\end{equation}
For a density $f=f(t,x,p)$, with $(t,x)$ suppressed in the notation below, the Bose--Einstein
collision operator is formally
\begin{align}\label{eq-qbe-collision}
Q_{\rm BE}(f)(p)
&=\int_{\R^3}\int_{\mathbb S^2}B(|p-p_*|,\omega)
\Big[f'f_*'(1+f)(1+f_*) -ff_*(1+f')(1+f_*')\Big]d\omega\,dp_*.
\end{align}  
Here $f=f(p)$, $f_*=f(p_*)$, $f'=f(p')$, $f_*'=f(p_*')$, and
\[
 p'=p-((p-p_*)\cdot\omega)\omega,
 \qquad
 p_*'=p_*+((p-p_*)\cdot\omega)\omega.
\]
The collisions conserve mass, momentum, and kinetic energy and satisfy the bosonic $H$-theorem. Thus,
for solution concepts for which the transport and collision identities are justified, one obtains a
free-energy inequality of the form required in Proposition~\ref{thm-dyn-stab}.

A dissipative alternative is the bosonic Vlasov--Fokker--Planck equation
\begin{equation}\label{eq-vfp}
 \partial_t F+p\cdot\nabla_xF-
 \nabla_x(\Phi+V_F)\cdot\nabla_pF
 =\nabla_p\cdot\big(T\nabla_pF+pF(1+F)\big),
 \qquad -\Delta V_F=4\pi\rho_F.
\end{equation}
For smooth positive densities, this equation is the Kramers-type gradient flow of the same free
energy in the momentum variable. Formally,
\begin{equation}\label{eq-vfp-diss}
 \frac{d}{dt}\mathcal F_T(F(t))
 =-\iint_{\R^6}
 \frac{|T\nabla_pF+pF(1+F)|^2}{F(1+F)}\,dpdx\leq0.
\end{equation}
This again gives the abstract free-energy inequality whenever the formula can be justified by the
chosen weak solution theory.

We emphasize that the stability result does not depend on a specific well-posedness
framework for \eqref{eq-vp}, \eqref{eq-boltzmann-poisson}, or \eqref{eq-vfp}. It is conditional on the
existence of a measure-valued curve in $\mathcal A$ satisfying the fixed mass constraint and the
free-energy inequality; the displayed nonlinear expressions are not used to define those operations
on arbitrary measures.

\subsection{Relative free energy}
Let $F_0$ be the minimizer at mass $M$, and let
\[
 u_0(x)=\Phi(x)+\widetilde V_{F_0}(x)+c
\]
be given by Theorem~\ref{prop-EL-minimizer}. Recall from
\eqref{jdef} that $j'(s)=\log\frac{s}{1+s}$ for $s>0$.
The Euler--Lagrange formula for the absolutely continuous part is equivalent to
\begin{equation}\label{eq-EL-jprime}
 \frac{|p|^2}{2}+u_0(x)+Tj'(F_0^a(x,p))=0
 \qquad\text{for }dpdx\text{-a.e. }(x,p).
\end{equation}
Moreover,
\begin{equation}\label{eq-singular-contact-dyn}
 \frac{|p|^2}{2}+u_0(x)=0
 \qquad F_0^s\text{-a.e.}
\end{equation}
Indeed, Theorem~\ref{prop-EL-minimizer} gives $F_0^s=\rho_0^s\otimes\delta_{p=0}$ and
$u_0=0$ $\rho_0^s$-a.e.

\begin{lemma}[Relative free-energy identity]\label{lem-relative-free-energy}
Let $F\in\mathcal A$ have the same mass as $F_0$. Then
\begin{equation}\label{eq-relative-identity}
 \mathcal F_T(F)-\mathcal F_T(F_0)=\mathcal D_T(F\mid F_0),
\end{equation}
where $\mathcal D_T$ is defined in \eqref{rel-distance-main}. In particular,
$\mathcal D_T(F\mid F_0)\geq0$.
\end{lemma}

\begin{proof}
Since $F$ and $F_0$ have the same total mass, adding the constant $c$ to the one-particle
energy does not change the linear part:
\[
 \iint_{\R^6}\left(\frac{|p|^2}{2}+\Phi\right)d(F-F_0)
 =\iint_{\R^6}\left(\frac{|p|^2}{2}+\Phi+c\right)d(F-F_0).
\]
The Coulomb part expands as
\begin{align*}
&\frac1{8\pi}\int |\nabla V_F|^2dx-
 \frac1{8\pi}\int |\nabla V_{F_0}|^2dx  \\
&\qquad =\frac1{4\pi}\int \nabla V_{F_0}\cdot\nabla(V_F-V_{F_0})\,dx
 +\frac1{8\pi}\int |\nabla(V_F-V_{F_0})|^2dx.
\end{align*}
By Proposition~\ref{prop-quasicont}, the cross term equals
\[
 \int_{\R^3}\widetilde V_{F_0}\,d(\rho_F-\rho_{F_0}).
\]
Therefore
\begin{align}\label{eq-free-energy-expanded}
\mathcal F_T(F)-\mathcal F_T(F_0)
&=\iint_{\R^6}\left(\frac{|p|^2}{2}+u_0(x)\right)d(F-F_0) \notag\\
&\quad +T\iint_{\R^6}\big(j(F^a)-j(F_0^a)\big)\,dpdx
 +\frac1{8\pi}\int |\nabla(V_F-V_{F_0})|^2dx.
\end{align}
We now split the first integral into the absolutely continuous and singular parts. On the absolutely
continuous part, \eqref{eq-EL-jprime} gives
\begin{align*}
&\iint\left(\frac{|p|^2}{2}+u_0\right)(F^a-F_0^a)\,dpdx
 +T\iint\big(j(F^a)-j(F_0^a)\big)\,dpdx \\
&\qquad =T\iint\big(j(F^a)-j(F_0^a)-j'(F_0^a)(F^a-F_0^a)\big)\,dpdx.
\end{align*}
On the singular part, \eqref{eq-singular-contact-dyn} implies
\[
 \iint\left(\frac{|p|^2}{2}+u_0\right)dF_0^s=0.
\]
Thus the remaining singular contribution is
\[
 \iint\left(\frac{|p|^2}{2}+u_0\right)dF^s.
\]
Substituting these identities into \eqref{eq-free-energy-expanded} proves \eqref{eq-relative-identity}.
Finally, $j$ is convex, so the Bregman term is nonnegative. The singular weight is nonnegative
q.e. by Theorem~\ref{prop-EL-minimizer}, and it is integrated against a spatial finite-energy
measure, which does not charge the exceptional set. The squared field term is nonnegative. Hence
$\mathcal D_T(F\mid F_0)\geq0$.
\end{proof}

\subsection{Proof of conditional variational stability}
\begin{proof}[Proof of Proposition~\ref{thm-dyn-stab}]
By Lemma~\ref{lem-relative-free-energy},
\[
 \mathcal D_T(F(t)\mid F_0)=\mathcal F_T(F(t))-\mathcal F_T(F_0).
\]
The free-energy inequality for the admissible curve gives
\[
 \mathcal F_T(F(t))-\mathcal F_T(F_0)
 \leq \mathcal F_T(F(0))-\mathcal F_T(F_0)
 =\mathcal D_T(F(0)\mid F_0),
\]
and the relative-functional stability follows.

It remains to justify the weak stability statement. We first record the compactness consequence of
the variational argument. If $F_n\in\mathcal A$ has mass $M$ and
\[
 \mathcal D_T(F_n\mid F_0)\to0,
\]
then, by Lemma~\ref{lem-relative-free-energy},
\[
 \mathcal F_T(F_n)\to\mathcal F_T(F_0)=\mathcal F_{\min}(T).
\]
Thus $\{F_n\}$ is a minimizing sequence. The compactness part in the proof of
Proposition~\ref{prop-existence-minimizer} gives, after passing to subsequences, weak-* convergence
in $\mathcal M(\R^6)$ to a minimizer. Since the minimizer is unique, every subsequential limit is
$F_0$, and hence
\[
 F_n\stackrel{*}{\rightharpoonup}F_0
 \qquad\text{in }\mathcal M(\R^6).
\]
Moreover, the field part in \eqref{rel-distance-main} gives directly
\[
 \|\nabla(V_{F_n}-V_{F_0})\|_{L^2}\to0.
\]

Now suppose the asserted weak stability fails. Then there exist a weak-* neighborhood $\mathcal U$ of
$F_0$, a number $\eta>0$, admissible free-energy curves $F_n(t)$ on intervals $[0,\tau_n]$, and
times $t_n\in[0,\tau_n]$ such that
\[
 \mathcal D_T(F_n(0)\mid F_0)\to0,
\]
but either $F_n(t_n)\notin\mathcal U$ or
\[
 \|\nabla(V_{F_n(t_n)}-V_{F_0})\|_{L^2}\geq\eta.
\]
The first part of the proposition gives
\[
 \mathcal D_T(F_n(t_n)\mid F_0)
 \leq \mathcal D_T(F_n(0)\mid F_0)\to0.
\]
Applying the compactness consequence to the sequence $F_n(t_n)$ yields
$F_n(t_n)\stackrel{*}{\rightharpoonup}F_0$ and
$\|\nabla(V_{F_n(t_n)}-V_{F_0})\|_2\to0$, a contradiction. This proves the proposition.
\end{proof}

\section*{Data availability statement}
No datasets were generated or analysed during the current study.

\section*{Conflict of interest}
The authors declare that they have no conflict of interest. 
The authors also declare that this manuscript has not been previously published.

\section*{Acknowledgment}
The work of G.-C. Bae is supported by the National Research Foundation of Korea (NRF) grant funded by the Korea government (MSIT) (No. 2021R1C1C2094843).
The work of J. Seok is supported by the National Research Foundation of Korea (NRF) grant funded by the Korea government (MSIT) (No. RS-2025-00562375).

\appendix 
\section{Examples of condensate geometry under radial confinement}

In this section, we consider two explicit examples of confining potentials, $\Phi(x)=\Phi(0)+\frac{\omega^2}{2}|x|^2$ and $\Phi(x)=|x|^4$.
In the absence of repulsive interaction, the condensate for these potentials is concentrated at the unique minimum of $\Phi$, namely at
$\delta_{x=0}\otimes\delta_{p=0}$. In contrast, the repulsive interaction
spreads the condensed spatial mass over a ball for
$\Phi(x)=\Phi(0)+\frac{\omega^2}{2}|x|^2$, and over a spherical shell for
$\Phi(x)=|x|^4$. These examples provide a more explicit illustration of
Theorem~\ref{prop-EL-minimizer}.
We define the density of the thermal cloud by
\[
R_T(u_0):=
\int_{\R^3}
\frac{1}{
	\exp\left(\frac{1}{T}\left(\frac{|p|^2}{2}+u_0\right)\right)-1
}\,dp,
\qquad u_0:=\Phi+V_{F_0}+c\ge0.
\]

\begin{lemma}\label{L.A1}
For the following confining potential 
\[
\Phi(x)=\Phi(0)+\frac{\omega^2}{2}|x|^2, \qquad 3\omega^2>4\pi R_T(0),
\]
let $F_0$ be the minimizer at temperature $T>0$ and mass $M>0$.
Assume in addition that $\rho_0^s\neq0$.
Then there exists $R_*>0$ such that
\[
\mathcal C:=\{x\in\R^3:u_0(x)=0\} = \overline{B_{R_*}(0)}.
\]
Moreover,
\[
d\rho_0^s(x)
=
\left(
\frac{3\omega^2}{4\pi}-R_T(0)
\right)
\mathbf 1_{\{|x|\le R_*\}} dx.
\]
Writing $u_0(x)=u(r)$ with $r=|x|$, the contact radius $R_*$ is
characterized by
\[
\begin{cases}
\displaystyle
u''(r)+\frac{2}{r}u'(r)
=
3\omega^2-4\pi R_T(u(r)),
& r>R_*,
\\[2mm]
u(R_*)=0,
\qquad
u'(R_*)=0,
\\[2mm]
\displaystyle
M
=
\omega^2R_*^3
+
4\pi\int_{R_*}^{\infty}
R_T(u(r))\,r^2\,dr .
\end{cases}
\]
\end{lemma}

\begin{proof}
Since $\Phi$ is radial, uniqueness of the minimizer implies that
$F_0$, $V_{F_0}$, and $u_0$ are radial. By
Theorem~\ref{prop-EL-minimizer},
\[
d\rho_0^s(x)
=
\left(
\frac{3\omega^2}{4\pi}-R_T(0)
\right)\mathbf 1_{\mathcal C}(x)\,dx,
\]
where we used $\Delta\Phi=3\omega^2$ and
$\rho_0^a=R_T(u_0)$. 

We next determine the shape of the contact set. We claim that the contact set is necessarily a ball centered at the origin. On the non-contact set $\{u_0>0\}$, the singular part vanishes and $\rho_{F_0}=R_T(u_0)$. 
Using $-\Delta V_{F_0}=4\pi\rho_{F_0}$, we obtain
\[
\Delta u_0 = \Delta \Phi - 4\pi \rho_{F_0}
=
3\omega^2-4\pi R_T(u_0)\qquad\text{for a.e. }r\in\{u_0>0\}.
\]
Writing $u_0(x)=u(r)$ and using $\Delta u_0= u''(r)+\frac2r u'(r)$, we get
\begin{align}\label{r2u'}
(r^2u'(r))'
=
r^2\bigl(3\omega^2-4\pi R_T(u(r))\bigr) \geq r^2(3\omega^2-4\pi R_T(0)) >0\qquad\text{for a.e. }r\in\{u>0\}
\end{align}
on every connected component of $\{u>0\}$, where we used the fact that $R_T(u)$ is decreasing in $u$ on $[0,\infty)$.

Since $\rho_0^s\neq0$ by assumption, the contact set is nonempty. We first claim that the origin belongs to the contact set, i.e., $0 \in \mathcal{C}$. Suppose, on the contrary, that the contact set has a positive inner radius $R_1>0$.
Then $u_0>0$ on $(0,R_1)$. Moreover, by Theorem \ref{prop-EL-minimizer}, $u_0 \in C^{1,1}_{loc}(\R^3)$, and since $u_0$ is radial, differentiability at the origin implies $u'(0)=0$. Integrating \eqref{r2u'} from $0$ to $r$ with $r< R_1$ implies that $u'(r) >0$, so that $u(r)$ is strictly increasing for $0\leq r< R_1$. Hence, $u(R_1)> u(0) >0$, which contradicts $u(R_1)=0$. Thus, the contact set must contain the origin.

Starting from $r=0$, let $R_*:=\sup\{r\geq0:\ u(s)=0\ \text{for every }0\leq s\leq r\}$ denote its outer radius. Since $u_0\in C^{1,1}_{loc}$ and
$u_0\equiv0$ on the contact region, we have $u_0(R_*)=u_0'(R_*)=0$.
For $r>R_*$, the same inequality \eqref{r2u'} yields that $u_0'(r)>0$ for $r>R_*$, so that $u(r)$ is strictly increasing. Hence, $u_0$ can never return to the
contact set $u(r)=0$. Therefore, $\mathcal C=\overline{B_{R_*}(0)}$.

Finally, on the contact set $0\leq r\leq R_*$ and outside the contact set $r\geq R_*$, we have
\begin{align*}
\begin{cases}
&\rho_0^a+\rho_0^s
=
R_T(0)
+
\frac{3\omega^2}{4\pi}-R_T(0)
=
\frac{3\omega^2}{4\pi}, \quad \mbox{on} \quad 0\leq r\leq R_*, \\
&\rho_{F_0}=R_T(u(r)), \quad \hspace{3.5cm} \mbox{on} \quad r\geq R_*.
\end{cases}
\end{align*}
Hence, the mass constraint yields
\[
\begin{aligned}
M
&=
4\pi\int_0^{R_*}
\frac{3\omega^2}{4\pi}r^2\,dr
+
4\pi\int_{R_*}^{\infty}
R_T(u(r))r^2\,dr
=
\omega^2R_*^3
+
4\pi\int_{R_*}^{\infty}
R_T(u(r))r^2\,dr.
\end{aligned}
\]
This completes the proof.
\end{proof}

\begin{lemma}\label{L.A2}
Let $\Phi(x)=|x|^4$, and let $F_0$ be the minimizer at temperature $T>0$ and mass $M>0$. Assume that $\rho_0^s\neq0$. Then, for $r_c:=\sqrt{\frac{\pi}{5}R_T(0)}$, there exist $0<r_c \leq R_1<R_2<\infty$ such that
\[
\operatorname{supp}\rho_0^s
=
\left\{x\in\R^3:R_1\le |x|\le R_2\right\}.
\]
Moreover,
\[
d\rho_0^s(x)
=
\left(
\frac{5}{\pi}|x|^2-R_T(0)
\right)
\mathbf 1_{\{R_1\le |x|\le R_2\}}\,dx.
\]
Furthermore, the contact radii $R_1,R_2$ and the radial profiles
$u_-,u_+$ satisfy
\[
\begin{cases}
\displaystyle
u_-''+\frac{2}{r}u_-'
=
20r^2-4\pi R_T(u_-), \quad u_-'(0)=u_-(R_1)=u_-'(R_1)=0, 
\quad 0<r<R_1,
\\[2mm]
\displaystyle
u_+''+\frac{2}{r}u_+'
=
20r^2-4\pi R_T(u_+), \quad u_+(R_2)=u_+'(R_2)=0,
\quad r>R_2, \\[2mm]
M
=
4R_2^5
+
4\pi\int_{R_2}^{\infty}
R_T(u_+(r))\,r^2\,dr.
\end{cases}
\]
\end{lemma}
\begin{proof}
Since $\Phi$ is radial, the uniqueness of the minimizer implies that
$F_0$, $V_{F_0}$, and $u_0$ are radial. We write $u_0(x)=u(r)$, where $r=|x|$. Since $\Delta\Phi(x)=20|x|^2$, Theorem~\ref{prop-EL-minimizer} gives
\[
d\rho_0^s(x)
=
\left(
\frac{1}{4\pi}\Delta\Phi(x)-R_T(0)
\right)
\mathbf 1_{\mathcal C}(x)\,dx
=
\left(
\frac{5}{\pi}|x|^2-R_T(0)
\right)
\mathbf 1_{\mathcal C}(x)\,dx.
\]
Since $\rho_0^s$ is a nonnegative measure, its support must be contained in
\[
\operatorname{supp}\rho_0^s
\subset
\{|x|\ge r_c\},
\qquad
r_c:=\sqrt{\frac{\pi}{5}R_T(0)}.
\]
Since $\rho_0^s\neq0$, Theorem~\ref{prop-EL-minimizer} implies that the contact set has nonempty interior. Let 
\[
R_1 := \inf \{r\geq0: u(r)=0\}, \quad R_2:=\sup\left\{r\ge R_1:u(s)=0\quad\text{for every } s\in [R_1,r]\right\}.
\]
As in the proof of Lemma~\ref{L.A1}, on the non-contact set we have
\[
(r^2u'(r))'
=
r^2\bigl(20r^2-4\pi R_T(u(r))\bigr)\qquad\text{for a.e. }r\in\{u>0\}.
\]

We first show that $R_1\ge r_c$. Suppose, to the contrary, that
$R_1<r_c$. Since $u(R_1)=0$, we have
\[
0=20r_c^2 - 4\pi R_T(0) > 20R_1^2 -4\pi R_T(0) = 20R_1^2 -4\pi R_T(u(R_1)).
\]
Since both $u$ and $R_T$ are continuous, there exists
$\delta>0$ such that $20r^2-4\pi R_T(u(r))<0$ for $R_1-\delta<r<R_1$. Hence,
\[
(r^2u'(r))' = r^2\bigl(20r^2-4\pi R_T(u(r))\bigr) <0, \quad \mbox{for a.e. } r\in (R_1-\delta,R_1).
\]
Since $u'(R_1)=0$, integrating from
$r$ to $R_1$ yields $u'(r)>0$. Therefore,
\[
u(R_1)-u(r)=\int_r^{R_1}u'(s)\,ds>0,
\]
which implies $u(r)<0$, contradicting $u\ge0$. Thus, $R_1\ge r_c$. In addition, since $u\in C_{\rm loc}^{1,1}$ and $u\equiv0$ on
$[R_1,R_2]$ by the definition of $R_2$, we have $u(R_2)=u'(R_2)=0$.

We next claim that $u$ cannot return to the obstacle after leaving the
contact set at $R_2$. Indeed, on the non-contact set $\{u>0\}$, as above,
\[
(r^2u'(r))'
=
r^2\bigl(20r^2-4\pi R_T(u(r))\bigr)\qquad\text{for a.e. }r\in\{u>0\}.
\]
Since $u'(R_2)=0$, integrating over $[R_2,r]$ for $r>R_2$ gives 
\[
\begin{aligned}
	r^2u'(r)
	&=
	\int_{R_2}^r
	s^2\bigl(20s^2-4\pi R_T(u(s))\bigr)\,ds > \int_{R_2}^r
	s^2\bigl(20r_c^2-4\pi R_T(0)\bigr)\,ds
	=0,
\end{aligned}
\]
for $r>R_2\ge R_1\ge r_c$. Consequently, $u'(r)>0$.
Thus, once $u$ leaves the obstacle at $R_2$, it is strictly
increasing and can never return to zero. Hence, no additional
contact component, and therefore no additional component of
$\operatorname{supp}\rho_0^s$, can occur outside $R_2$. Since $R_1$ is the first contact radius, it follows that the condensed part is supported on a
single spherical shell, $\operatorname{supp}\rho_0^s=\{x\in\R^3:R_1\le |x|\le R_2\}$.

We now derive the corresponding relations for the radial profiles. On the non-contact set $\{u>0\}$, the equation
\[
u''+\frac{2}{r}u'
=
20r^2-4\pi R_T(u), \qquad\text{a.e. on }(0,R_1)\cup(R_2,\infty),
\]
gives the equations for $u_-$ and $u_+$. Since $u_0\in C_{\mathrm{loc}}^{1,1}(\R^3)$ is radial,
differentiability at the origin yields $u_-'(0)=0$.
Furthermore, $u=0$ on the contact shell $[R_1,R_2]$ and
$u\in C^1$, so the smooth-fit conditions at the free boundaries are $u_-(R_1)=u_-'(R_1)=0$ and $u_+(R_2)=u_+'(R_2)=0$.

It remains to derive the mass constraint. Since the total spatial density on the shell is $\rho_{F_0}(r)=\rho_0^a(r)+\rho_0^s(r)=\frac{5}{\pi}r^2$, the mass constraint is equivalent to 
\begin{align}\label{M=}
	M= 4\pi \int_0^{R_1} R_T(u_-(r))r^2dr + 4\pi \int_{R_1}^{R_2} \frac{5}{\pi}r^4 dr + 4\pi \int_{R_2}^{\infty} R_T(u_+(r))r^2dr.
\end{align}
Using the properties of the radial Newtonian potential, we can simplify the first integral. We define
\[
m(r):=
4\pi\int_0^r \rho_{F_0}(s)s^2\,ds.
\]
For the radial Newtonian potential $V_{F_0}(x):=\int_{\mathbb R^3}\frac{1}{|x-y|}\,d\rho_{F_0}(y)$, we have
\[
V_{F_0}(r)
=
\frac{m(r)-m(0)}{r}
+
\int_{\{|y|>r\}}\frac{1}{|y|}\,d\rho_{F_0}(y), \qquad V_{F_0}'(r)=-\frac{m(r)}{r^2}.
\]
Since $u(r) = r^4 +V_{F_0}(r)+c$ and $u'(r)=4r^3+V_{F_0}'(r)$, we have
\[ 
u'(r)
=
4r^3-\frac{m(r)}{r^2}.
\]
Substituting the condition $u'(R_1)=0$ gives $m(R_1)=4R_1^5$, which is equivalent to $4\pi\int_0^{R_1}
R_T(u_-(r))r^2\,dr
=
4R_1^5$. Combining this with the identity $4\pi \int_{R_1}^{R_2} \frac{5}{\pi}r^4 dr = 4R_2^5-4R_1^5$ in \eqref{M=} gives the desired mass constraint.
\end{proof}

\end{document}